\documentclass[12pt]{article}
\usepackage{authblk}
\usepackage{blindtext}
\usepackage{amsmath,amssymb,amsthm}
\usepackage[colorlinks,citecolor=blue]{hyperref}
\usepackage[utf8]{inputenc}
\usepackage{tikz}
\usepackage{tikz-cd}
\usetikzlibrary{arrows,shapes,chains,matrix,positioning,scopes,cd,patterns}
\usetikzlibrary{decorations.pathmorphing}
\usetikzlibrary{decorations.pathreplacing,calligraphy}
\usetikzlibrary{backgrounds,calc,positioning}
\usetikzlibrary{arrows.meta}
\usepackage[all]{xy}
\definecolor{darkblue}{rgb}{0.0,0,0.7} 
\newcommand{\darkblue}{\color{darkblue}} 
\definecolor{darkred}{rgb}{0.7,0,0} 
\usepackage{hyperref}
\usepackage{cleveref}
\usepackage{mathabx}
\usepackage{pigpen}
\pdfmapline{=dictsym DictSym <dictsym.pfb}
\pdfmapline{=pigpen <pigpen.pfa}

\newtheorem{theorem}{Theorem}[section] 
\newtheorem{proposition}[theorem]{Proposition} 
 
\newtheorem{corollary}[theorem]{Corollary} 
\newtheorem{lemma}[theorem]{Lemma} 
\newtheorem{question}[theorem]{Question}

\theoremstyle{definition}
\newtheorem{definition}[theorem]{Definition}
\newtheorem{remark}[theorem]{Remark}

\newcommand{\po}{\ar@{}[dr]|{\text{\pigpenfont R}}}
\newcommand{\pb}{\ar@{}[dr]|{\text{\pigpenfont J}}}

\newcommand{\Tc}{\mathcal{T}}

\newcommand{\C}{\mathcal{C}}
\newcommand{\Sc}{\mathcal{S}}
\newcommand{\Lc}{\mathcal{L}}
\newcommand{\Rc}{\mathcal{R}}
\newcommand{\Cov}{\operatorname{Cov}_{\downarrow}}

\newcommand{\Inc}{\operatorname{Inc}}
\newcommand{\Dec}{\operatorname{Dec}}
\newcommand{\Trs}{\operatorname{Trs}}
\newcommand{\Wfs}{\operatorname{Wfs}}
\newcommand{\Lss}{\operatorname{Lss}}
\newcommand{\ideal}{\operatorname{Ideal}}
\newcommand{\Rel}{\operatorname{Rel}}
\newcommand{\Con}{\operatorname{Con}}

\newcommand{\W}{\mathcal{W}}

\newcommand{\con}{\operatorname{con}}
\newcommand{\maxcov}{\operatorname{mcov}_{\downarrow}}

\newcommand{\Tr}{\operatorname{Tr}}

\newcommand{\Ls}{\operatorname{Ls}}
\newcommand{\jirr} {\mathrm{j\text{-}Irr}}
\newcommand{\pairs}{\operatorname{Pairs}}
\crefname{figure}{figure}{figures}
\Crefname{figure}{Figure}{Figures}

\newcommand{\emp}[1]{\emph{\darkblue #1}} 

\newcommand{\darkred}{\color{darkred}} 
\newcommand{\defn}[1]{\emph{\darkred #1}}

\title{On the lattice of the weak factorization systems on a finite lattice}
\author[1, 2]{Yongle Luo}
\author[2]{Baptiste Rognerud}

\affil[1]{School of Mathematical Sciences, Nanjing Normal University, 210023, Nanjing, China}
\affil[2]{Institut de Mathématiques de Jussieu, Paris Rive Gauche (IMJ-PRG), Campus des Grands Moulins,
Université de Paris - Boite Courrier 7012, 8 Place Aurélie Nemours,
75205 PARIS Cedex 13, France}
\affil[ ]{\textit {yongle.luo@imj-prg.fr,baptiste.rognerud@imj-prg.fr}}

\date{}

\usepackage{array}
\newcolumntype{P}[1]{>{\centering\arraybackslash}p{#1}}

\begin{document}

\maketitle


\begin{abstract}
We consider the lattice of all the weak factorization systems on a given finite lattice. We prove that it is semidistributive, trim and congruence uniform. We deduce a graph theoretical approach to the problem of enumerating transfer systems. As an application we find a lower bound for the number of transfer systems on a boolean lattice. 
\end{abstract}

\tableofcontents

\thispagestyle{empty}

\section{Introduction}

A \defn{weak factorization system} on a category is a pair $(\Lc,\Rc)$ of subclasses of the morphisms of the category satisfying a list of natural conditions (see \cref{def:wfs}). One can order the weak factorization systems on a given category by setting $(\Lc,\Rc) \preceq (\Lc',\Rc')$ if $\Rc \subseteq \Rc'$ and $\Lc' \subseteq \Lc$. An interval in the poset of weak factorization system was called a \defn{premodel structure} by Barton in \cite{barton}. A \defn{model structure}, in the sense of Quillen, is a premodel structure which satisfies an additional property. In other words, the model structures are particular intervals in this poset. If one is interested in classifying model structures, it is tempting to start by classifying the weak factorization systems. Such a classification seems to be difficult to obtain in general but a lot of results have recently been obtained (see e.g. \cite{blumberg2024homotopical} for a very nice introductory article on the subject) in the particular case where the category is a \emp{finite lattice}.

In that case, the two classes $\Lc$ and $\Rc$ of a weak factorization systems satisfy a very short list of properties: they are subposets of $(L,\leq)$ which are respectively closed under pushouts (joins in lattice terminology) and pullbacks (meets). The class $\Rc$ is called a \defn{transfer system} and we call the class $\Lc$ a \defn{left saturated set}. We can order transfer systems by inclusion of relations and similarly, we order the left saturated sets by reverse inclusion. It is classical that the classes $\Lc$ and $\Rc$ determine each other. So, the projections onto each component induce injective morphisms of posets from the poset of weak factorization systems to the poset of transfer systems and of left saturated sets. By \cite[Theorem 4.13]{franchere2021self}, the three posets are in fact isomorphic. Moreover by \cite[Proposition 3.7]{franchere2021self} they are lattices where the meet is given by \emp{intersection of the transfer systems} and the join is given by the \emp{intersection of left saturated sets}. Surprisingly \emp{transfer systems} have a different origin. There is a notion of $G$-transfer systems (see \cref{def:gts}) where $G$ is a finite group, which can be used to classify the so-called $N_{\infty}$ operads up to homotopy (see \cite[Theorem 3.6]{franchere2021self} and the references above it). In the particular case where the group $G$ is abelian (or every subgroup is normal) a $G$-transfer system is nothing but a transfer system on its lattice of subgroups.   This gives another motivation for the classification of the weak factorization systems on a given lattice. Many results in that direction have been recently obtained: for the cyclic $p$-groups see \cite{roitzheim2022n}, for the symmetric group $S_3$ or the quaternion group $Q_8$ see \cite{rubin}, the elementary abelian $p$-groups of rank two see \cite{bao2023transfer}, the Dihedral group $D_{p^n}$ see \cite{balchin2022combinatorics}, see also \cite{blumberg2024homotopical} for more references.

The lattice of subgroups of a cyclic $p$-group is a total order and it has been proved by  Roitzheim, Barnes and Balchin that the lattice of transfer systems is in this case isomorphic to the \emp{Tamari lattice} (see \cite[Theorem 25]{roitzheim2022n}, see also \cref{sec:total_order} for an alternative proof). The Tamari lattices are particular orderings of the Catalan numbers and they appear throughout mathematics. In representation theory, they are known to be the lattices of \emp{torsion pairs} of the path algebras of equioriented quivers of type $A$ (see e.g. \cite{thomas_tamari}). 

The notion of torsion pairs in an \emp{abelian category} has been introduced by Dickson in \cite{dickson1966torsion} and has received a lot of attention recently due to its relation with the \emp{$\tau$-tilting theory} of Adachi, Iyama and Reiten \cite{tau_tilting}. The lattice properties of the lattice of torsion pairs in the category of finitely generated modules over an artinian algebra have been extensively studied in \cite{DIRRT}. The fact that the Tamari lattice appears at the intersection of these two theories is the starting point for this article and our results can be summarized by saying that from the lattice point of view, the weak factorization systems of finite lattices and the torsion pairs of artinian algebras are extremely similar. The main difference lies in the fact there are generally \emp{infinitely many torsion pairs}. A more detailed comparison of the properties is given in \cref{sec:dico}. 

Let us be more precise on the different results of this article, we refer to \cref{fig:commutative_square} and \cref{fig:square} for an illustration in the case where the lattice $L$ is a commutative square. A lattice $(L,\land,\lor)$ is said to be \defn{distributive} if the meet distributes over the join and vice-versa. As main example, given a poset $(P,\leq)$ we consider $\ideal(P)$ the set of \defn{order ideals} of $P$ (downwards closed subsets). The inclusion of ideals turns $\ideal(P)$ into a poset and since both the intersection and union of ideals are ideals, it is a lattice, and it is distributive by the `distributive laws' in set theory. Moreover, by the Birkhoff \emp{representation theorem}, any finite distributive lattice $L$ is isomorphic to $\ideal(P)$ where $P$ is the subposet of $L$ containing the \emp{join-irreducible elements} of $L$. For our setting, this notion is too strict and rigid, indeed as first result we have a characterization of the distributive lattices of weak factorization systems. 

\begin{proposition}
Let $L$ be a finite lattice. Then the lattice of weak factorization systems on $L$ is \emp{distributive} if and only if $|L|\leq 2$.
\end{proposition}
One should note that for more general categories $\C$ there are other examples of distributive lattices of weak factorization systems. This is for example the case of the \emp{category of sets} (see  the poset in \cite{Antolin}). 

\emp{Semidistributive lattices} are a generalization of the distributive lattices introduced by J\'{o}nsson \cite{jonsson1961} who was inspired by Whitman’s solution to the word problem for free lattices involving  the existence of a nice canonical form. We refer to \cite[Chapter 3]{gratzerlattice} for more details. We say that a lattice $(L,\land,\lor)$ is \defn{semidistributive} if $\forall a,b,c\in L$
 \[ 
 a \vee (b \wedge c) = ( a\vee b ) \wedge (a\vee c)  \hbox{ whenever } a\vee b = a\vee c,
 \]
 and 
 \[
 a\wedge (b \vee c) = (a\wedge b) \vee (a \wedge c) \hbox{ whenever } a\wedge b = a\wedge c.
 \]

For weak factorization systems we have: 
\begin{theorem}\label{thm:semidistributive}
Let $(L,\leq)$ be a finite lattice. Then the lattice of weak factorization systems on $L$ is a \emp{semidistributive} lattice.
\end{theorem}
In the semidistributive lattices we have the existence of a so-called \defn{canonical join decomposition} for their elements. We refer to \cref{sec:canonical_join} for explicit definitions. Loosely speaking any element of a semidistributive lattice can be written in a unique way as the join of \emp{join-irreducible elements}. This is a lattice version of the fundamental theorem of arithmetic which says that any integer can be written in a unique way as product of prime numbers. 

The join-irreducibles in the poset of transfer systems are particularly easy to describe. We denote by $\Rel(L)$ the set of all \emp{relations of $(L,\leq)$}. That is the set of $(a,b)$ such that $a \leq b$. We denote by $\Rel^*(L)$ the set of the non-trivial relations on $L$ i.e., the relations $(a,b)$ with $a\neq b$. Then, we prove that sending a relation $(a,b)$ to the smallest transfer system $\Tr(a,b)$ containing $(a,b)$ is a bijection between the non-trivial relations in $L$ and the join-irreducible elements in the lattice of transfer systems on $L$. For two relations $(a,b)$ and $(c,d)$ we write $(a,b) \sqsubset (c,d)$ if there is a pullback diagram
\[
\xymatrix{
a \pb \ar[r] \ar[d] & b \ar[d]\\
c\ar[r] & d.
}
\]
\begin{proposition}\label{pro:join-irreducible}
Let $(L,\leq)$ be a finite lattice. The poset of join-irreducible elements in the lattice of weak factorization systems on $L$ is isomorphic to $(\Rel^*(L) ,\sqsubset)$. 
\end{proposition}
We refer to \cref{fig:square} for an illustration in the simple case where $L$ is a commutative square. 

As a corollary, we obtain the following result, which answers a question in \cite[remark 4.5]{bao2023transfer}:
\begin{proposition}\label{pro:isomorphism}
    Let $L$ and $L'$ be two finite lattices with isomorphic lattices of weak factorization systems. Then $L$ and $L'$ are isomorphic. 
\end{proposition}
More precisely, we show that $L\setminus\{0\}$ (where $0$ is the least element) is isomorphic to the largest principal ideal of $(\Rel^*(L),\sqsubset)$.

For two relations $(a,b)$ and $(c,d)$ in $\Rel^*(L)$, we write $(a,b) \boxslash (c,d)$ if $(a,b)$ lifts on the left $(c,d)$ (see \cref{sec:wfs} for more details). Moreover, for every $S\subseteq \Rel^*(L)$ we say that $S$ is an \defn{elevating set } if for every $(a,b) \in S$ and $(c,d)\in S$ we have $(a,b) \boxslash (c,d)$ and $(c,d) \boxslash (a,b)$. The existence of canonical join representations in the semidistributive lattice of weak factorization systems, implies the following result.
\begin{theorem}\label{thm:semibricks}
Let $(L,\leq)$ be a finite lattice. There is a bijection between the set of weak factorization systems on $L$ and the set of elevating subsets of $\Rel^*(L)$. 
\end{theorem}
The bijection is moreover explicit: if $S$ is an elevating set, then the corresponding transfer system is the smallest transfer system containing $S$. The inverse bijection is more subtle: in any semidistributive lattice there is a \emp{labelling of the cover relations} by the join-irreducible elements of the lattice. For transfer systems a cover relation $\Rc_1 \lessdot \Rc$ is labelled by the unique non-trivial relation of $\Rc \cap \,^{\boxslash} \Rc_1$ (see \cref{lem:join label}). Given a transfer system $\Rc$, the corresponding elevating set is the set consisting of the labels of all the lower cover relations $\Rc'\lessdot \Rc$. 

The set of all the canonical join-representations of the elements of a semidistributive lattice $L$ can be naturally viewed as a \emp{simplicial complex}. It has been proved by Emily Barnard \cite{barnard_canonical} that this simplicial complex is a \defn{flag}. This means that it is the \defn{clique complex} of its $1$-skeleton. In our setting we call it the \defn{elevating graph} of the lattice $L$. This is the \emp{graph} whose \emp{vertices} are the non-trivial relations of $(L,\leq)$ and there is an \emp{edge} between $(a,b)$ and $(c,d)$ if and only if $(a,b) \boxslash (c,d)$ and $(c,d) \boxslash (a,b)$. For the commutative square this graph is illustrated in \cref{fig:square}. Recall that a clique is a complete (induced) subgraph. It follows that there are as \emp{many transfer systems} on $L$ as there are \emp{cliques in the elevating graph} of $L$. For example, there are $10$ cliques in the elevating graph of the commutative square: the empty graph, $5$ vertices and $4$ edges, so there are $10$ transfer systems on this lattice. To summarize we have:

\begin{figure}
\begin{center}
\scalebox{0.7}{
\tikzstyle{block} = [rectangle, fill=red!20,
    text width=7em, text centered, rounded corners, minimum height=1em, node distance=3.5cm]
\tikzstyle{line} = [draw, very thick, color=black, -latex']
\begin{tikzpicture}[scale=4, node distance = 2cm, auto]
\node [block] (0) {
$\xymatrix{
 & 3 &  \\
1 &  & 2 \\
 & 0 & 
}$
};
\node [block,above left of=0, node distance=6cm,fill=blue!20] (1) {
$\xymatrix{
 & 3 &  \\
1 &  & 2 \\
 & 0 \ar@{->}[lu] & 
}$
};
\node [block,above right of=0, node distance=6cm,fill=blue!20] (2) {
$\xymatrix{
 & 3 &  \\
1 &  & 2 \\
 & 0 \ar@{->}[ru] & 
}$
};
\node [block,above right of=1, node distance=6cm] (3) {
$\xymatrix{
 & 3 &  \\
1 &  & 2 \\
 & 0 \ar@{->}[ru] \ar@{->}[lu] & 
}$
};
\node [block,above left of=1, node distance=6cm,fill=blue!20] (4) {
\xymatrix{
 & 3 &  \\
1 &  & 2 \ar@{->}[lu] \\
 & 0 \ar@{->}[lu] & 
}
};
\node [block,above right of=2, node distance=6cm,fill=blue!20] (5) {
\xymatrix{
 & 3 &  \\
1 \ar@{->}[ru] &  & 2 \\
 & 0\ar@{->}[ru] & 
}
};
\node [block,above of=3, node distance=6cm,fill=blue!20] (6) {
\xymatrix{
 & 3 &  \\
1 &  & 2\\
 & 0\ar@{->}[ru] \ar@{->}[lu] \ar@{->}[uu] & 
}
};
\node [block,above left of=6, node distance=6cm] (7) {
\xymatrix{
 & 3 &  \\
1 &  & 2 \ar@{->}[lu] \\
 & 0\ar@{->}[ru] \ar@{->}[lu] \ar@{.>}[uu] & 
}
};
\node [block,above right of=6, node distance=6cm] (8) {
\xymatrix{
 & 3 &  \\
1 \ar@{->}[ru] &  & 2 \\
 & 0 \ar@{->}[ru] \ar@{->}[lu] \ar@{.>}[uu] & 
}
};
\node [block,above right of=7, node distance=6cm] (9) {
\xymatrix{
 & 3 &  \\
1 \ar@{->}[ru] &  &2 \ar@{->}[lu] \\
 & 0 \ar@{->}[ru] \ar@{->}[lu] \ar@{.>}[uu] & 
}
};
\path [line] (0) -- (1) node[midway,blue,anchor=north,xshift=-0.5cm] {$(0,1)$};
\path [line] (1)--(3) node[midway,blue,anchor=north,xshift=0.5cm] {$(0,2)$};
\path [line] (1)--(4) node[midway,blue,anchor=north,xshift=-0.5cm] {$(2,3)$};
\path [line] (2)--(5) node[midway,blue,anchor=north,xshift=0.5cm] {$(1,3)$};
\path [line] (2)--(3) node[midway,blue,anchor=north,xshift=-0.5cm] {$(0,1)$};
\path [line] (0)--(2) node[midway,blue,anchor=north,xshift=0.5cm] {$(0,2)$};
\path [line] (3)--(6) node[midway,blue,anchor=north,xshift=0.5cm] {$(0,3)$};
\path [line] (6)--(7) node[midway,blue,anchor=north,xshift=-0.5cm] {$(2,3)$};
\path [line] (6)--(8) node[midway,blue,anchor=north,xshift=0.5cm] {$(1,3)$};
\path [line] (4)--(7) node[midway,blue,anchor=north,xshift=0.5cm] {$(0,2)$};
\path [line] (5)--(8) node[midway,blue,anchor=north,xshift=-0.5cm] {$(0,1)$};
\path [line] (8)--(9)node[midway,blue,anchor=north,xshift=-0.5cm] {$(2,3)$};
\path [line] (7)--(9)node[midway,blue,anchor=north,xshift=0.5cm] {$(1,3)$};
\end{tikzpicture}
}
\end{center}
\caption{Lattice of transfer systems of the commutative square. The transfer systems in blue are join-irreducible. The labelling  of the edges is the join labelling. 
}\label{fig:commutative_square}
\end{figure}
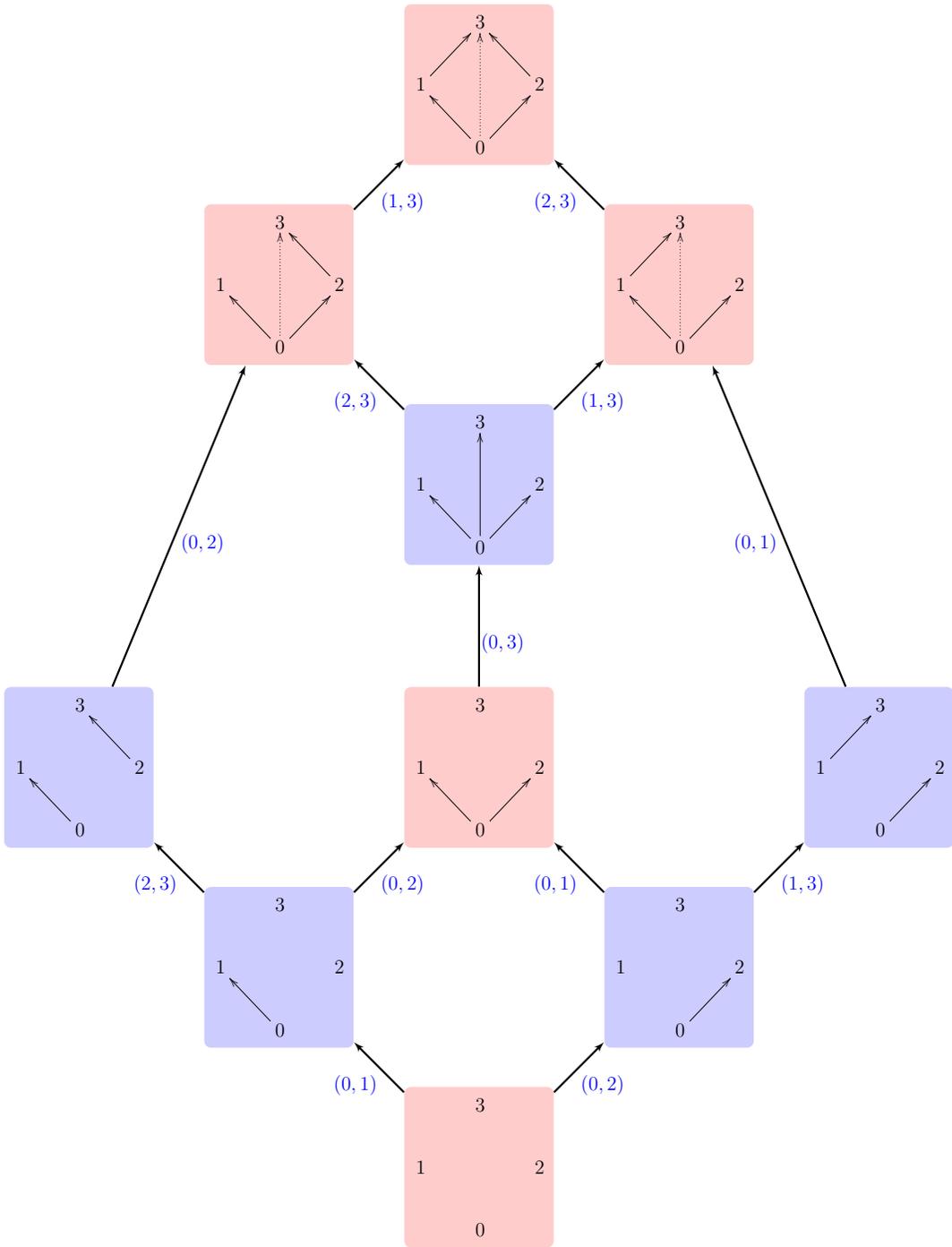

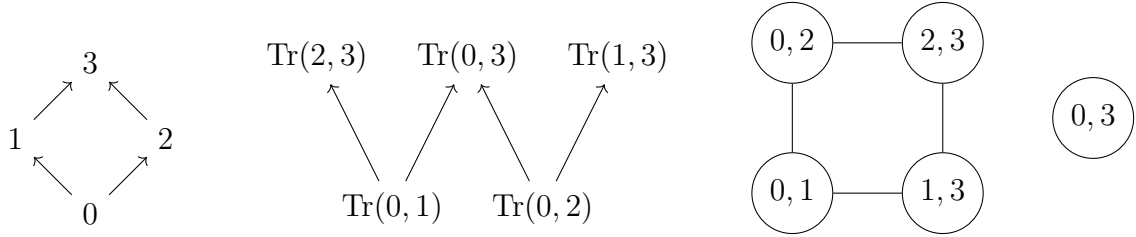
\begin{figure}[h!]
\centering 
\begin{tikzpicture}
\node (A) at (0,0) {$0$};
 \node (B) at (-1,1) {$1$};
\node (C) at (1,1) {$2$};
\node (D) at (0,2) {$3$};
\draw[->] (A) --(B);
\draw[->] (B)--(D);
\draw[->] (A)--(C);
\draw[->] (C)--(D);
\end{tikzpicture}
\qquad 
 \begin{tikzpicture}
\node (A) at (1,0) {$\Tr(0,1)$};
 \node (B) at (3,0) {$\Tr(0,2)$};
\node (C) at (0,2) {$\Tr(2,3)$};
\node (D) at (2,2) {$\Tr(0,3)$};
\node (E) at (4,2) {$\Tr(1,3)$};
\draw[->] (A) --(C);
\draw[->] (A) --(D);
\draw[->] (B) --(D);
\draw[->] (B) --(E);
\end{tikzpicture}
\qquad 
\begin{tikzpicture}
\node[shape=circle,draw=black] (A) at (0,0) {$0,1$};
 \node[shape=circle,draw=black] (B) at (0,2) {$0,2$};
\node[shape=circle,draw=black] (C) at (2,0) {$1,3$};
\node[shape=circle,draw=black] (D) at (2,2) {$2,3$};
\node[shape=circle,draw=black] (E) at (4,1) {$0,3$};
\draw (A) --(B)--(D)--(C)--(A) ;
\end{tikzpicture}
\caption{From left to right: the lattice $L$, the poset of join-irreducible transfer systems, and the elevating graph of $L$. One can check that each clique of the graph corresponds to the set of labellings of the lower covers of a transfer system.}\label{fig:square}
\end{figure}
\begin{theorem}\label{thm:cliques}
Let $(L,\leq)$ be a finite lattice. There is a bijection between the set of weak factorization systems on $L$ and the cliques of its elevating graph. 
\end{theorem}
We deduce from this result a relatively efficient algorithm for counting the weak factorization systems on a (not too big) lattice. It is known that counting cliques in a graph is a difficult task (it is $\sharp P$-complete) however the elevating graph is much smaller than the lattice of transfer systems. For example if $L = \mathcal{P}([4])$ is the boolean lattice with $16$ elements, there are $5389480$ weak factorization systems but the elevating graph of $L$ has only $65$ vertices and $1474$ edges. 

Moreover any subgraph of a clique is also a clique, so if there is a clique of size $n$ in a graph,then there are at least $2^n$ cliques. In our setting the \emp{maximal size of a clique}, is the maximal size of an elevating set, hence it is the \defn{maximal number of lower cover relations} of a transfer system, which we denote by $\maxcov(L)$. 

\begin{corollary}
    Let $(L,\leq)$ be a finite lattice. There are at least $2^{\maxcov(L)}$ weak factorization systems on $L$.
\end{corollary}

The lattice of weak factorization system is \emp{not regular} in general and the number $\maxcov(L)$ seems difficult to determine in general, but in the special case of the boolean lattice we obtain the following lower bound:

\begin{proposition}
    Let $L$ be the boolean lattice of the subsets of $[n]$. Then, there are at least $2^{a_n}$ transfer systems on $L$, where
    \begin{enumerate}
        \item If $n$ is odd, $a_n= \sum_{j=0}^{\frac{n-1}{2}} \binom{n}{j} \binom{n-j}{n-2j}$.
        \item If $n$ is even, $a_n=\sum_{j=1}^{\frac{n}{2}} \binom{n}{j}\binom{n-j}{n+1-2j}.$
    \end{enumerate}
    
\end{proposition}

Our next result involves the notion of \defn{trim} lattice which was introduced by Hugh Thomas in \cite{thomas_trim} as a non-graded generalization of \emp{distributive lattice}, we refer to \cref{sec:trim} for the definition and for more details. The notions of trimness and semidistributivity are independent of each other: there are examples of trim lattices that are not semidistributive and conversely. However, it has been proved in \cite{TW} that a semidistributive lattice which is also \defn{extremal} is trim. Here by an extremal, we mean a lattice whose length is equal to its number of join-irreducible elements and also to the number of meet-irreducible elements. 

\begin{theorem}
Let $(L,\leq)$ be a finite lattice. Then the lattice of weak factorization systems on $L$ is a \emp{trim lattice}.
\end{theorem}
As a comparison, \emp{the lattice of torsion pairs} of artinian algebras is in general \emp{not a trim lattice}. Trimness occurs for example for the lattice of torsion pairs of the path algebra of \emp{Dynkin type} and for incidence algebras of posets of finite representation type (see \cite{TW}). For a counterexample, the lattice of torsion pairs of a \emp{preprojective algebra} of Dynkin type is isomorphic to the weak Bruhat ordering of the corresponding Weyl group (\cite{mizuno2014classifying}) and this is not a trim lattice. 

Trim lattices have an important distributive sublattice which is called the \defn{spine}. It consists of the elements on the chains of maximal length. We prove that a weak factorization system $(\Lc,\Rc)$ is on the spine if and only $\Lc \cup \Rc = \Rel(L)$. In other words if and only if every morphism is either in $\Lc$ or in $\Rc$. If $(a,b)$ is a non-trivial relation of $L$, then it can be seen as the interval $\{ x\in L\ |\ a\leq x \leq b\}$ of the poset. These intervals can be naturally ordered by setting $(a,b)\preceq (c,d)$ if and only if $a\leq c$ and $b\leq d$. We obtain the following representation of the spine:

\begin{proposition}
Let $(L,\leq)$ be a finite lattice. Then the spine of $\Trs(L)$ is isomorphic to $\ideal\big((\Rel^*(L),\preceq)^{op}\big)$.
\end{proposition}

The following results concern the \defn{congruences} of the lattice of weak factorization systems. Recall that an equivalence relation compatible with the meet and the join of the lattice is called a congruence. The set $\Con(L)$ of all congruences on $L$ is also a lattice when ordered by refinement and it is known to be a \emp{distributive} lattice. The Birkhoff theorem tells us that $\Con(L) = \ideal(P)$ where $P$ is the poset of join-irreducible congruences. Roughly speaking a lattice $L$ is \defn{congruence uniform} when the irreducible congruences are in bijection with the irreducible elements of $L$, we refer to \cref{sec:congruence_uniform} for a more precise statement. By labelling the cover relations in the lattice of weak factorization systems (a consequence of semidistributivity), we are able to prove the following result. 

\begin{theorem}\label{thm:cong_unif}
Let $(L,\leq)$ be a finite lattice. Then the lattice of weak factorization systems on $L$ is a \emp{congruence uniform} lattice.
\end{theorem}
The non-trivial relations $a\leq b$ on $L$, when seen as intervals of the poset $L$ are also naturally ordered by containment $\subseteq$. Then we have:
\begin{theorem}\label{thm:lat_of_cong}
Let $(L,\leq)$ be a finite lattice. Then the lattice of congruences of the lattice of weak factorization systems on $L$ is isomorphic to $\ideal\big((\Rel^{*}(L),\subseteq)^{op}\big)$. 
\end{theorem}

Finally, let us return to the $G$-transfer systems. They coincide with the transfer systems on the lattice $\operatorname{Sub}(G)$ of the subgroups of $G$, when the group is abelian, but  in general, this is a proper \emp{sublattice}. We prove in \cref{sec:gts} that it inherits the properties of the lattice of transfer systems on $\operatorname{Sub}(G)$. 

\begin{theorem}\label{thm:gts}
Let $G$ be a finite group. Then the lattice of $G$-transfer systems is \emp{semidistributive}, \emp{trim} and \emp{congruence uniform}. 
\end{theorem}

\paragraph{Acknowledgement}
The first author acknowledges the financial support provided by the China Scholarship Council (CSC) under grant number 202306860058. This article was finished when the second author was visiting the CCM Unam Morelia and he thanks Gerardo Raggi for his warm hospitality. Both of the authors would like to thank Hipolito Treffinger for valuable information. 

\section{Weak factorization systems and transfer systems}
In this short section we recall the definitions and basic results on weak factorization systems and transfer systems on a finite lattice $L$. 
\subsection{Weak factorization systems}\label{sec:wfs}

Let $\C$ be a category. A morphism $f$ of $\C$ is said to \defn{lift on the left} a morphism $g$ of $\C$ if for every commutative square
\[
\xymatrix{
A\ar[r]\ar[d]_{f} & X\ar[d]^{g}\\
B\ar[r] & Y
}
\]
there exists a lift $h : B\to X \in \C$ making the resulting triangles commute. In this case we write $f \boxslash g$. If $\Sc$ is a class of morphisms in $\C$, we use the following notation
\[
\Sc^{\boxslash} = \{ g\in \operatorname{Mor}(\C)\ |\ f\boxslash g \ \forall f\in \Sc\},
\]
\[
\,^{\boxslash}\Sc = \{ f\in \operatorname{Mor}(\C)\ |\ f\boxslash g \ \forall g\in \Sc\}.
\]

It is clear that $\Sc \subseteq \,^{\boxslash}\Tc$ if and only if $\Tc \subseteq \Sc^{\boxslash}$ and when it holds we write $\Sc \boxslash \Tc$.

\begin{definition}\label{def:wfs}
 A \defn{weak factorization system} on $\C$ is a pair $(\Lc,\Rc)$ of subclasses of the morphisms of $\C$ such that:
 \begin{enumerate}
 \item Every morphism $f\in \C$ can be factored as $f = pi$ where $i\in \Lc$ and $p\in \Rc$.
 \item $\Lc \boxslash \Rc$. 
 \item $\Lc$ and $\Rc$ are closed under retracts. 
 \end{enumerate}
 \end{definition}
 Here we say that a class of maps $\mathcal{M}$ is \defn{closed under retracts} if for every commutative diagram
 \[
 \xymatrix{
 A \ar[r]^{i}\ar[d]^{f} & X \ar[r]^{j} \ar[d]^{g} & A\ar[d]^{f}\\
 B \ar[r]^{k} & Y\ar[r]^{l} & B
 }
 \]
 such that $ji = id_A$, $lk = id_B$ and $f\in \mathcal{M}$, then $g\in \mathcal{M}$. Now, we have a useful alternative definition for the weak factorization systems. 
  
\begin{proposition}\label{pro:defts2}
A pair $(\Lc,\Rc)$ of subclasses of the morphisms of $\C$ is a weak factorization system for $\C$ if and only if
\begin{enumerate}
\item Every morphism $f\in \C$ can be factored as $f = pi$ where $i\in \Lc$ and $p\in \Rc$.
\item[2'.] $\Lc = \,^{\boxslash}R$.
\item[3'.] $\Rc = \Lc^{\boxslash}$. 
\end{enumerate}
\end{proposition}
\begin{proof}
See for example \cite{adamek2002weak} or \cite[Proposition 14.1.13]{may}.
\end{proof}

For a given category $\C$ there are always at least \emp{two weak factorization systems}: $\Lc = \operatorname{Mor}(\C)$ and $\Rc = Iso(\C)$, the class of all isomorphisms of $\C$ and conversely $\Lc = Iso(\C)$ and $\Rc = \operatorname{Mor}(\C)$. Moreover, set theoretical issues aside, the class of weak factorization systems has a \emp{natural structure of partial order}: 
\begin{center} $(\Lc ,\Rc) \preceq (\mathcal{L}',\mathcal{R}')$ if $\Rc \subseteq \mathcal{R}'$ and $\mathcal{L}' \subseteq \Lc$. \end{center}

 Note that one inclusion implies the other. As discussed above this poset, denoted by $\Wfs(\mathcal{C})$, has always a greatest element and a smallest element. 
 
 A \defn{premodel structure} for $\C$ is an \emp{interval} $(\Lc,\Rc) \preceq (\Lc',\Rc')$ in the poset of weak factorization systems. It is a \defn{model structure} if the class $\mathcal{W} = \Rc \circ \Lc'$ satisfies the two out of three property: if two of $f,g, f\circ g$ are in $\W$, then the three are. In this case, the elements of $\mathcal{W}$ are called the \defn{weak equivalences}, the elements of $\Rc'$ the \defn{fibrations} and the elements of $\Lc$ the \defn{cofibrations}. 

Since the left part and right part of a weak factorization system determine each other, we need to understand which class of maps can appear as the left or the right part of a weak factorization system. We quote \cite[Definition 14.1.7]{may}
 
 \begin{definition}
 Let $\mathcal{M}$ be a class of morphisms of the category $\C$. Then $\mathcal{M}$ is \defn{left saturated} if the class $\mathcal{M}$
 \begin{enumerate}
 \item  is closed under \emp{compositions} and contains all the \emp{isomorphisms} of $\C$;
 \item is closed under \emp{retracts};
 \item is closed under arbitrary \emp{coproducts} (in the morphisms category);
 \item is closed under \emp{pushouts};
 \item is closed under \emp{transfinite compositions}.
 \end{enumerate}
 \end{definition}
Here by closed under pushouts we mean that for any \emp{pushout} diagram
\[
\xymatrix{
A\po \ar[r]\ar[d]^{i} & B\ar[d]^{f} \\
C \ar[r] & D
}
\]
if $i \in \mathcal{M}$, then $f\in \mathcal{M}$. There is a dual definition of \defn{right saturated} class of morphisms. 

\begin{proposition}
Let $\C$ be a category and $(\Lc,\Rc)$ be a weak factorization system for $\C$. Then $\Lc$ is \emp{left saturated} and $\Rc$ is \emp{right saturated}. 
\end{proposition}
\begin{proof}
This is \cite[Proposition 14.1.8]{may}. 
\end{proof}

\subsection{Transfer systems}
We now consider the case where $\C$ is the category of a finite poset. Moreover since we want our categories to have limits and colimits we will assume that the poset is a \emp{lattice}. Concretely if $(L,\leq)$ is a finite lattice, it can be viewed as a category whose objects are the elements of $L$. There is a unique morphism from $x$ to $y$ if and only if $x\leq y$. Since a poset is reflexive, there is a unique morphism from $x$ to $x$, which is nothing but the identity morphism at $x$. Moreover, the transitivity of the relation implies that our construction is a category. The antisymmetry of the relation implies that the only isomorphisms in this category are the identity morphisms. In particular there are no non-trivial retracts and we can remove the third item in the definition of weak factorization systems. A set of morphisms in this category is nothing but a \emp{subrelation} of $(L,\leq)$. To say that such a class contains the \emp{isomorphisms} means that the relation is \emp{reflexive} and to say that it is closed under \emp{compositions} means that this is a \emp{transitive relation}. A subrelation of an antisymmetric relation is always antisymmetric. In other words, a composition closed set of morphisms that contains the isomorphism is a \emp{subposet} of $(L,\leq)$. The definition of \emp{right saturated} class of morphisms becomes:

\begin{definition}\label{def:ts}
Let $(L,\leq)$ be a lattice. A \defn{transfer system} $\lhd$ for $L$ is a relation of partial ordering on $L$ such that:
\begin{enumerate}
\item $i\ \lhd \ j$ implies $i \leq j$. 
\item $i\ \lhd \ k$ and $j \leq k$ implies $(i \land j)\ \lhd \ j$. 
\end{enumerate}
\end{definition}
Inclusion of relations naturally induces a poset structure on the set of transfer systems on $L$ and we denote this poset by $\Trs(L)$. 
\begin{theorem}\label{thm:transfer_system}
Let $(L,\leq)$ be a lattice. The map sending a transfer system $\mathcal{R}$ to $(\,^{\boxslash}\mathcal{R},\mathcal{R})$ is an isomorphism between the poset of transfer systems and the poset of weak factorization systems. 
\end{theorem}
\begin{proof}
This is \cite[Theorem 4.13]{franchere2021self}.

\end{proof}
\begin{remark}
It is clear that we can dualize this proof to obtain an isomorphism between the poset of weak factorization systems on $L$ and the poset of left saturated sets ordered by reverse inclusion. 
\end{remark}

\subsection{$G$-equivariant topology}

The notion of $N_\infty$-operad has been introduced in \cite{blumberg2015operadic} as an equivariant analogue of $E_\infty$-operads. As the authors are not expert in this area, we simply refer to the \cite[Section 3]{franchere2021self} for more informations about this notion and for the details. There is a notion of weak equivalences of $N_{\infty}$-operads, hence one can define the homotopy category of these operads. It turns out that this category is closely related to the notion of transfer systems.

\begin{definition}\label{def:gts}
Let $G$ be a finite group. We denote by $\operatorname{Sub}(G)$ the poset of the subgroups of $G$. Then a \defn{$G$-transfer system} is a subposet $\lhd$ of the inclusion on $\operatorname{Sub}(G)$ such that:
\begin{enumerate}
\item If $H \lhd K$ and $g\in G$, then $gHg^{-1} \lhd gKg^{-1}$.
\item If $H\lhd K$ and $M \leq K$, then $M\cap H \lhd M$.
\end{enumerate}
\end{definition}
The poset $\big(\operatorname{Sub}(G),\leq \big)$ is a lattice and by forgetting the first item in \cref{def:gts} we see that a $G$-transfer system is in particular a \emp{transfer system} on $\big(\operatorname{Sub}(G),\leq \big)$ in the sense of \cref{def:ts}. Moreover, when the group $G$ is abelian or when all the subgroups of $G$ are normal (e.g., the quaternion group $Q_8$) a $G$-transfer system is nothing but a transfer system.

\begin{theorem}
Let $G$ be a finite group. The homotopy category of $N_\infty$-operads is equivalent to the poset of $G$-transfer systems (viewed as a category). 
\end{theorem}
\begin{proof}
This theorem involves the work of various groups of mathematicians, see \cite[Theorem 3.6]{franchere2021self} and the references above it. 
\end{proof}

It follows that understanding the transfer systems on finite lattices is important for \emp{two reasons}: first they are the building blocks of the weak factorization systems, which is the first step towards the model structures on the finite lattices. Secondly, they classify the $G$-equivariants analogues of the $E_\infty$-operads up to equivalences. For the second case, we should focus on the lattices of subgroups of given finite groups. 
\section{Lattice of weak factorization systems}
In this section we investigate the properties of the poset of weak factorization systems on a finite lattice $L$. Since this poset is isomorphic to the poset of transfer systems and to the poset of left saturated sets, we will allow ourself to move between these three notions. 
\subsection{Basic properties}
We start by a reformulation of \cref{thm:transfer_system}. For a finite lattice $(L,\leq)$ we denote by $\Wfs(L)$ the poset of weak factorization systems and by $\Trs(L)$ the poset of transfer systems. We also denote by $\Lss(L)$ the poset of \emp{left saturated sets} ordered by \defn{reverse inclusion}. Here a left saturated set is nothing but a subposet of $L$ which is \emp{closed under pushouts}.
\begin{lemma}\label{lem:duality}
Let $(L,\leq)$ be a finite lattice. The map sending a left saturated set to its opposite, is an anti-isomorphism of posets between $\Lss(L)$ and $\Trs(L^{op})$.
\end{lemma} 

Moreover, if $\Rc$ is a transfer system, then $\Lc = \, ^{\boxslash}\Rc$ is a left saturated set. Similarly if $\Lc$ is a left saturated set, then $\Rc = \Lc ^{\boxslash}$ is a transfer system. By \cref{thm:transfer_system} and its dual, the pair $(^{\boxslash}\Rc, \Rc)$ is a weak factorization system so by \cref{pro:defts2} we have $\Rc = \big(^{\boxslash}\Rc\big)^{\boxslash}$. Similarly we have $\Lc = \, ^{\boxslash}\big(\Lc^{\boxslash})$. Finally, since the maps $\Rc \mapsto \, ^{\boxslash} \Rc$ and $\Lc \mapsto \Lc^{\boxslash}$ reverse the inclusions, the posets $\Trs(L)$ and $\Lss(L)$ are isomorphic. Hence we have:
\begin{proposition}\label{pro:anti-iso}
Let $(L,\leq)$ be a finite lattice. There is a commutative diagram of isomorphisms of posets:
\[
\xymatrix{
& \Wfs(L)\ar[dl]_{\pi_1} \ar[dr]^{\pi_2}&\\
\Lss(L)\ar@/^1pc/[rr]^{-^{\boxslash}} & & \Trs(L)\ar@/^1pc/[ll]^{\,^{\boxslash}-}
}
\]
where $\pi_i$ denotes the canonical projections for $i=1,2$.
\end{proposition}
Our next step is to recall that the poset of weak factorization systems is a lattice (see \cite[Proposition 3.7]{franchere2021self}). Since the intersection of two transfer systems is a transfer system and there are only finitely many transfer systems this is clear. However we will need a good description of the join in this poset. 

When $(L,\leq)$ is a finite lattice, we denote by $\Rel(L) = \{ (a,b) \in L^{2}\ |\ a\leq b\}$ the set of relations of $L$. We also denote by $\Delta(L) = \{(x,x) \ |\ x\in L\}$, the diagonal of $L$. We will use the notation $a\leq b$ or $a\to b$ to indicate that $(a,b)\in \Rel(L)$. Moreover we denote by $\Rel^{*}(L) = \Rel(L) \setminus \Delta(L)$ the set of non-trivial relations in $L$. 

 If $S\subseteq \Rel(L)$ we denote by $S^{pb} = \{ (x\land w, w) \in L^{2}\ |\ w\leq y \hbox{ and } (x, y) \in S\}$, the \emp{closure of $S$ under pullbacks}. Similarly, we denote by $S^{po} = \{(w,y\lor w)\in L^{2}\ |\ x\leq w \hbox{ and } (x, y) \in S\}$, the \emp{closure of $S$ under pushouts}. We also denote by $\Tr(S)= (S^{pb})^{tc}$ the set obtained by taking the reflexive and transitive closure of the pullback closure of $S$. 

\begin{lemma}\label{lem:smallest_ts}
Let $(L,\leq)$ be a finite lattice. Then
\begin{enumerate}
\item Let $S\subseteq \Rel(L)$, then $\Tr(S)$ is the \emp{smallest transfer system} containing $S$. 
\item If $S = \{ (a,b) \}$ for $(a, b) \in \Rel^*(L)$, then
\[ \Tr(S) = S^{pb}\cup \Delta(L) = \{ (a\land c,c) | c\leq b\} \cup \Delta(L).\] 
\item Let $\Rc_1$ and $\Rc_2$ be two transfer systems for $L$. Then 
\[ \Tr(\Rc_1 \cup \Rc_2 ) = (\Rc_1 \cup \Rc_2)^{tc}.\] 
\end{enumerate}
\end{lemma}
\begin{proof}
\begin{enumerate}
\item This is \cite[Theorem A.2]{rubin}. Since it is easy, we give the argument. For every $x\leq y \in S$, note that $(x,y) = (x\land y, y)$, we have $(x,y)\in S^{pb}$. That is, $S \subseteq \Tr(S)$. By the construction of $\Tr(S)$, it is clear that it is reflexive and transitive. Since $\Tr(S)$ is a subrelation of $(L,\leq)$, it is also antisymmetric. Moreover, it follows from the \emp{pasting laws for pullbacks} (or elementary manipulations of meets in a lattice) that $\Tr(S)$ is closed under pullbacks.  Hence, we proved $\Tr(S)$  is a transfer system containing $S$. If $\Rc$ is any transfer system containing $S$, then it is closed under pullbacks. Particularly, it contains all pullbacks of $S$, that is, $S^{pb} \subseteq \Rc$. Since $\Rc$, is reflexive and transitive, we get $\Tr(S) \subseteq \Rc$. 
\item The elements of $S^{pb}$ are relations of the form $a\land c_{i} \to c_{i}$ with $c_{i}\leq b$, so we can compose them only when $c_i = a\land c_j$ for some $c_i, c_j \leq b$. Note that 
\[
a \land c_i = a \land (a \land c_j) = a \land c_j = c_i.
\]
That is, $a\land c_i \to c_i$ is trivial. 
\item This is because $(\Rc_1 \cup \Rc_2)^{pb} = \Rc_1 \cup \Rc_2$. 
\end{enumerate}
\end{proof}
There is of course the dual result for the \defn{smallest left saturated set} containing $S$, which we denote by $\Ls(S)$. We obtain the following result. 
\begin{proposition}\label{pro:wfs_lattice}
Let $(L,\leq)$ be a finite lattice. Then $\Wfs(L)$ is a lattice and if $(\Lc_1,\Rc_1)$ and $(\Lc_2,\Rc_2)$ are two weak factorization systems, then
\[
(\Lc_1,\Rc_1) \land (\Lc_2,\Rc_2) = (\Ls(\Lc_1 \cup \Lc_2), \Rc_1 \cap \Rc_2),
\]
and 
\[
(\Lc_1,\Rc_1) \lor (\Lc_2,\Rc_2) = (\Lc_1 \cap \Lc_2, \Tr(\Rc_1\cup \Rc_2) ). 
\]
\end{proposition}
\subsection{Cover relations and join-irreducible elements}

We recall that an element $j\in L$ is \defn{join-irreducible} if for every finite subset $A\subseteq L$ such that $j = \lor A$ we have $j \in A$.  Equivalently $j$ is join-irreducible if and only if it covers a unique element $j_*$. Our next task is to determine the join-irreducibles of the lattice of transfer systems. To do so, we start by a discussion on the cover relations in the poset of transfer systems. We use the notation $\Rc_1 \lessdot \Rc_2$ if $\Rc_1 \subset \Rc_2$ is a cover relation, that is $\Rc_1\neq \Rc_2$ and if $\Rc_1 \subseteq \Rc \subseteq \Rc_2$ we have $\Rc = \Rc_1$ or $\Rc = \Rc_2$. 

If $(a,b) \in \Rel^{*}(L)$, then $\Tr(a,b)$ is a transfer system on $L$. Moreover, we have $\Tr(a,b)\subseteq \Tr(c,d)$ if and only if $(a,b) \in \Tr(c,d)$. By \cref{lem:smallest_ts}, this is equivalent to the existence of a pullback diagram:
\[
\xymatrix{
b\land c = a \pb \ar[r]\ar[d] & b \ar[d] \\
c \ar[r] & d 
}
\]
It follows that the map $(a,b) \mapsto \Tr(a,b)$ is an injection from $\Rel^{*}(L)$ to $\Trs(L)$. The inclusion of transfer systems induces a poset relation on $\Rel^{*}(L)$ that we denote by $(a,b) \sqsubset (c,d)$. 

Let $\Rc$ be a transfer system on $L$. We denote by $\Rc^{c}$ the set of all \emp {cover relations} in $\Rc$. Moreover, we say that $(a,b)\in \Rc^c$ is \defn {maximal} in $\Rc^c$, if for any $(c, d) \in \Rc^c$ such that $(a,b) \sqsubset (c,d)$, we have $(a,b) =(c,d)$.

\begin{lemma}\label{lem:lower_ts}
Let $\Rc$ be a transfer system on $L$ and $(a,b)\in \Rc^c$. We denote by $\Rc_{(a,b)} = \Rc \setminus \{ (u,v)\in \Rc \ |\ u\neq v \hbox{ and } (a,b) \sqsubset (u,v) \}.$ Then $\Rc_{(a,b)}$ is a \emp{transfer system} on $L$. Moreover, if $(a,b)$ is maximal in $\Rc^c$, then $\Rc_{(a,b)} \lessdot \Rc$.
\end{lemma}
\begin{proof}
The relation $\Rc_{(a,b)}$ is reflexive and antisymmetric. Let $(u,v)\in \Rc_{(a,b)}$ and $(v,w)\in \Rc_{(a,b)}$. If $(u,w) \notin \Rc_{(a,b)}$, we have a pullback diagram
\[
\xymatrix{
u\land b = a \pb \ar[rr]\ar[d] && b\ar[d]\\
u \ar[r] & v \ar[r] & w
}
\]
which we complete into a diagram where all squares are pullbacks
\[
\xymatrix{
u\land b = a\pb \ar[r]\ar[d] & v \land b\pb \ar[d]\ar[r]& b\ar[d]\\
u \ar[r]_{\in \Rc} & v \ar[r]_{\in \Rc} & w
}
\]
The arrows in the bottom are in $\Rc$ and since $\Rc$ is closed under pullback, we have $(a,v\land b) \in \Rc$ and $(v\land b, b) \in \Rc$. Since $(a,b)$ is a cover relation in $\Rc$, we have $a = v \land b$ or $v\land b = b$. In the first case we obtain $(v,w)\notin \Rc_{(a,b)}$ and in the second case we obtain $(u,v)\notin \Rc_{(a,b)}$, which are both contradictions.

We also need to prove that $\Rc_{(a,b)}$ is closed under pullback, so we consider $(u,v)\in \Rc_{(a,b)}$ and $(u\land w,w)$ an element obtained by taking the pullback of $(u,v)$ with $w\leq v$. If $(u\land w,w) \notin \Rc_{(a,b)}$ we have $(a,b) \sqsubset (u\land w,w)$, and the following commutative diagram where the two small squares are pullbacks:
\[
\xymatrix{
a\pb \ar[r] \ar[d] & u \land w\pb \ar[r]\ar[d] & u \ar[d] \\
b \ar[r] & w \ar[r] & v
}
\]
The pasting law for pullbacks implies that $(u,v)\notin\Rc_{(a,b)}$.

It remains to prove that when $(a,b)$ is maximal in $\Rc^c$, the inclusion $\Rc_{(a,b)} \subseteq \Rc$ is a cover relation. Let $\Rc^{'}$ be any transfer system on $L$ such that $\Rc_{(a,b)} \subsetneq \Rc^{'} \subseteq \Rc$. Note that we assume that $\Rc_{(a,b)} \neq \Rc^{'}$, so there exists $(x,y)\in \Rc^{'}$ such that $(x,y)\notin \Rc_{(a,b)}$. By the definition of $\Rc_{(a,b)}$ we have $(a,b) \sqsubset (x,y)$. Since $\Rc^{'}$ is closed under pullbacks, we have $(a,b)\in \Rc^{'}$. 

Let $(u,v)\in \Rc^{c}$, then if $(a,b) \not\sqsubset (u,v)$, then $(u,v)\in \Rc_{(a,b)} \subseteq \Rc'$. If $(a,b) \sqsubseteq (u,v)$, then $(u,v) = (a,b) \in \Rc'$. So all the cover relations of $\Rc$ are in $\Rc'$ hence $\Rc' = \Rc$. 

\end{proof}

\begin{lemma}\label{lem:step1}
Let $\Rc_{1} \subseteq \Rc$ be two transfer systems on $L$. Then we have
\begin{enumerate}
\item  $\Rc_{1} =\Rc$ if and only if $\Rc \cap \,^{\boxslash} \Rc_{1} = \Delta(L)$.
\item  If $(a,b)\in \Rc^{c} \setminus \Rc_{1}$, then $(a,b) \in \Rc \cap \,^{\boxslash} \Rc_{1}$. 

\end{enumerate}
\end{lemma}
\begin{proof}
\begin{enumerate}
\item It is clear that $\Rc \cap \,^{\boxslash} \Rc = \Delta(L)$. Conversely let $(a,b) \in \Rc$. Since $(\,^{\boxslash}\Rc_{1},\Rc_{1})$ is a weak factorization system on $L$ we have a factorization $a \to c \to b$ with $(a,c)\in \,^{\boxslash}\Rc_{1}$ and $(c,b) \in \Rc_1$. Since $\Rc$ is a transfer system, the pullback diagram
\[
\xymatrix{
a \pb \ar[r] \ar[d] & c \ar[d]\\
a \ar[r] &b
}
\]
implies that $(a,c) \in \,^{\boxslash}\Rc_{1} \cap \Rc = \Delta(L)$, so we have $a=c$ and $(a,b) \in \Rc_{1}$.

\item We consider a cover relation $(a,b)$ in $\Rc$ which is not in $\Rc_{1}$ and let $(u,v) \in \Rc_{1}$ be such that there is a commutative diagram
\[
\xymatrix{
a \ar[r]\ar[d] & u \ar[d]\\
b \ar[r] & v}
\]
We complete it as the following diagram
\[
\xymatrix{
a \ar[r]^{\in \Rc} \ar[dr]_{\in \Rc} & b\land u \pb \ar[r] \ar[d]^{\in \Rc_{1}} & u\ar[d]^{\in \Rc_{1}}\\
& b\ar[r] &v. 
}
\]
Here we used the fact that the right square is a pullback to obtain $(b\land u,b) \in \Rc_{1} \subseteq \Rc$. Moreover, the pullback
\[
\xymatrix{
a \pb \ar[r] \ar[d] & b\land u \ar[d] \\
a \ar[r]_{\in \Rc} & b
}
\]
implies that $(a,b\land u) \in \Rc$. Since $(a,b)$ is a cover relation in $\Rc$ we have $b \land u = a$ or $b\land u = b$. In the first case, we obtain $(a,b)\in \Rc_{1}$ and this is a contradiction. Hence we have $b \land u = b$, so $b\leq u$ and $(a,b) \in \,^{\boxslash} \Rc_{1}$. 

\end{enumerate}
\end{proof}
\begin{lemma}\label{lem:cov_relations}
Let $\Rc_1 \subseteq \Rc$ be two transfer systems on $L$. Then, the following are equivalent.
\begin{enumerate}
\item $\Rc_1 \subseteq \Rc$ is a cover relation in $\Trs(L)$. 
\item $\Rc^{c} \setminus \Rc_1$ contains only one relation $(a,b)$, and $\Rc_1 = \Rc_{(a,b)}$.
\item $\Rc \cap \,^{\boxslash} \Rc_{1}$ contains only one non-trivial relation.
\end{enumerate}
\end{lemma}
\begin{proof}\
\begin{itemize}
\item[ $ 1. \Longrightarrow 2.$] Since $\Rc_{1} \lessdot \Rc$, there is a non-trivial relation $(a,b) \in \Rc \setminus \Rc_1$. This relation is a composition of cover relations in $\Rc$ and if all these cover relations are in $\Rc_{1}$, then the composition itself is in $\Rc_1$. So we can assume that $(a,b)\in \Rc^c$. By \cref{lem:lower_ts}, we have $\Rc_{(a,b)} \subseteq \Rc$ and $\Rc_{(a,b)} \neq \Rc$. Suppose that there is $(u,v) \in \Rc_1$  such that $(u,v) \notin \Rc_{(a,b)}$. Then $(a,b) \sqsubset (u,v)$, and there is a pullback diagram
\[
\xymatrix{
a\pb \ar[r]\ar[d] & b \ar[d]\\
u\ar[r]_{\in \Rc_1} & v
}
\]
Note that $\Rc_1$ is a transfer system and it is closed under pullback, so $(a,b) \in \Rc_1$ which is a contradiction. Hence we have $\Rc_{1} \subseteq \Rc_{(a,b)} \subsetneq \Rc$. Since $\Rc_{1} \lessdot \Rc$, we have $\Rc_{1} = \Rc_{(a,b)}$. If there is another relation $(c,d)\in \Rc^{c} \setminus \Rc_{1}$, we obtain similarly $\Rc_{1} = \Rc_{(c,d)}$. So we have $\Rc_{(a,b)} = \Rc_{(c,d)}$, hence the sets $\{(u,v) \in \Rc\ |\ (a,b) \sqsubset (u,v) \}$ and $\{ (u,v) \in \Rc\ |\ (c,d) \sqsubset (u,v) \}$ are equal. In particular we have $(a,b) \sqsubset (c,d) \sqsubset (a,b)$, so $(a,b) = (c,d)$. 
\item[ $ 2. \Longrightarrow 3.$] Since there is $(a,b) \in \Rc^{c} \setminus \Rc_{1}$, the second item of \cref{lem:step1} implies that $(a,b)\in \Rc \cap \,^{\boxslash}\Rc_1$.  Conversely, let $(c,d)\in \Rc \cap \,^{\boxslash} \Rc_{1}$ with $c\neq d$. Since $\Rc_{1} \cap\,^{\boxslash} \Rc_{1} = \Delta(L)$ we have $(c,d) \notin \Rc_1$. It remains to prove that $(c,d)$ is a cover relation in $R$, that is, $(c,d)\in \Rc^{c}$. If not we suppose $(c, d)$ is a composition of $(c, e)$ and $(e, d)$, where $(c,e) \in \Rc$ and $(e, d) \in \Rc^c$. We have a pushout diagram:
\[
\xymatrix{
c\po \ar[r]\ar[d]_{\,^{\boxslash}\Rc_{1}\ni} & e\ar[d] \\
d\ar[r] & d 
}
\]
Since $^{\boxslash}\Rc_{1}$ is a left saturated set, it is closed under pushout, so $(e,d) \in \Rc\cap ^{\boxslash}\Rc_{1}$. In particular $(e,d) \in \Rc^{c} \setminus \Rc_1$,  hence it is equal to $(a,b)$. Finally $(c,d) \notin \Rc_{1} = \Rc_{(a,b)}$ implies $(a,b) \sqsubset (c,d)$, so we have a pullback diagram
\[
\xymatrix{
a\pb \ar[d] \ar[rr] && b\ar[d]\\
c \ar[r] & a \ar[r] & d=b
}
\]
so $a\leq c \leq a$, and $(c,d) = (a,b)$. 
\item[ $ 3. \Longrightarrow 1.$] Let $(a,b)$ be the unique non-trivial relation in $\Rc \cap \,^{\boxslash}\Rc_{1}$. Certainly, we have $\Rc_1 \neq \Rc$. Let $\Rc'$ be a transfer system such that $\Rc_{1} \subseteq \Rc' \subsetneq \Rc$. By the first item of \cref{lem:step1}, there is a non-trivial relation $(c,d) \in \Rc \cap \,^{\boxslash} \Rc'$. Since $\Rc_{1} \subseteq \Rc'$ we have  $\,^{\boxslash} \Rc' \subseteq \,^{\boxslash} \Rc_{1}$, so $(c,d) \in \Rc \cap \,^{\boxslash}\Rc_{1}$. Hence $(c,d) = (a,b)$, and it follows $(a,b) \in \, ^{\boxslash}\Rc^{'}$. If $\Rc_1 \neq \Rc'$, then there is a relation $(e,f) \in \Rc' \cap \,^{\boxslash}\Rc_1 \subseteq \Rc \cap \,^{\boxslash}\Rc_1$. Hence $(e,f) = (a,b)$, and it follows $(a,b) \in \Rc^{'}$.  Thus, we have $(a,b)\in \Rc^{'} \cap \, ^{\boxslash}\Rc^{'}$, that is, $(a,b)$ is trivial, which is a contradiction. 
\end{itemize}
\end{proof}

\begin{remark}
Let $\Rc$ and $\Rc_1$ be two transfer systems such that $\Rc_1 \lessdot \Rc$. It follows from the proof $2. \Longrightarrow 3.$ of \cref{lem:cov_relations} that the unique relation in $\Rc^{c}\setminus \Rc_1$ is exactly the unique non-trivial relation in $\Rc \cap \,^{\boxslash}\Rc_{1}$. Conversely, we denote by $(a,b)$ the unique non-trivial relation in $\Rc \cap \,^{\boxslash}\Rc_{1}$. Then $(a,b)$ is a cover relation in $\Rc$ and hence $(a,b)\in \Rc^{c} \setminus \Rc_{1}$. Indeed $(a,b)$ has a composition $a \to c\to b$ with $(a,c)\in \Rc$ and $(c,b) \in \Rc^{c}$. Since $^{\boxslash}\Rc_{1}$ is a left saturated system we have $(c,b)\in \Rc \cap \,^{\boxslash}\Rc_{1}$. So we have $(c,b)=(a,b)$.  
\end{remark}

Let $\Rc$ be a transfer system on $L$. We denote by $\Cov(\Rc)$ the set of all transfer systems covered by $\Rc$. We also denote by $\Rc^{c}_M$ the set of all maximal relations in $(\Rc^{c},\sqsubset)$. It has been proved in \cref{lem:lower_ts} that there is a map $\Phi$ from $\Rc^{c}_{M}$ to $\Cov(\Rc)$ defined by sending each $(a,b)$ in $\Rc^{c}_{M}$ to $\Rc_{(a,b)}$ in $\Cov(\Rc)$. Then we have the following:

\begin{proposition}\label{prop:lower_cover}
   The map $\Phi$ defined above is a bijection from $\Rc^{c}_{M}$ to $\Cov(\Rc)$.
\end{proposition}
\begin{proof}
    The fact that $\Phi$ is an injection has been seen in the proof \cref{lem:cov_relations}. For any $\Rc_{i} \in \Cov(\Rc)$, using \cref{lem:cov_relations}, there is a unique relation $(a_i, b_i)$ in $\Rc^{c}\setminus \Rc_{i}$. We have to see that $(a_i,b_i)$ is maximal in $(\Rc^{c},\sqsubset)$. If there is a cover relation $(u,v)$ in $\Rc^{c}$ such that $(a_i,b_i) \sqsubset (u, v)$, then we claim that $(u,v)\notin \Rc_i$. Otherwise it follows form the pullback diagram that $(a_i, b_i)\in \Rc_{i}$, which is a contradiction. 
\[
\xymatrix{
a_{i}\pb \ar[r]\ar[d] & b_{i} \ar[d]\\
u\ar[r]_{\in \Rc_i} & v
}
\]

 So $(u,v)$ is also in $\Rc^c \setminus \Rc_i$. Note that there is only one relation in $\Rc^{c}\setminus \Rc_{i}$. Then we have $(u,v)=(a_i, b_i)$ and hence $(a_i, b_i)$ is maximal in $\Rc^c$. It follows from the second item of \cref{lem:cov_relations} again that $\Phi(a_i,b_i)= \Rc_{(a_i, b_i)}=\Rc_{i}$. So we proved that $\Phi$ is sujective. 
\end{proof}
 
We are now ready to prove \cref{pro:join-irreducible}.
\begin{proof}[Proof of \cref{pro:join-irreducible}]

Let $\Rc$ be a join-irreducible element in $\Trs(L)$. It covers a unique element that we denote by $\Rc_{*}$. Since $\Rc_{*} \lessdot \Rc$, by \cref{lem:cov_relations}, there exists a unique relation $(a,b) \in \Rc^c \setminus \Rc_*$. So $\Tr(a,b)$, the smallest transfer system containing $(a,b)$, satisfies $\Tr(a,b) \subseteq \Rc$ and $\Tr(a,b) \not\subseteq \Rc_{*}$. Hence we have $\Rc_{*} \subsetneq \Tr(a,b) \lor \Rc_* \subseteq \Rc$. Hence $\Rc = \Rc_{*} \lor \Tr(a,b)$. Since $\Rc$ is join-irreducible and $\Rc \neq \Rc_{*}$, we have $\Rc = \Tr(a,b)$. 

Conversely if $(a,b) \in \Rel^*(L)$, we have to see that $\Tr(a,b)$ is a join-irreducible transfer system. Using \cref{lem:lower_ts}, we consider the transfer system $\Tr(a,b)_{*} := \Tr(a,b)_{(a,b)}$. It is obtained by removing from $\Tr(a,b)$ the relations $(u,v)\in \Tr(a,b)$ such that $(a,b)\sqsubset (u,v)$. But $(u,v)\in \Tr(a,b)$ implies that $(u,v) \sqsubset (a,b)$. Hence $\Tr(a,b)_* = \Tr(a,b) \setminus \{(a,b)\}$. By the second item of \cref{lem:cov_relations}, we have a cover relation $\Tr(a,b)_* \lessdot \Tr(a,b)$. If $\Rc$ is a transfer system such that $\Rc \subsetneq \Tr(a,b)$, then $(a,b)\notin \Rc$, so $\Rc \subseteq \Tr(a,b)_*$. In other words, $\Tr(a,b)$ covers a unique element $\Tr(a,b)_*$, hence it is join-irreducible. 
\end{proof}

\begin{proof}[Proof of \cref{pro:isomorphism}]

    Our proof is organized as follows. If $\Trs(L) \cong \Trs(L')$, then their posets of join-irreducible elements $(\Rel^*(L),\sqsubset)$ and $(\Rel^*(L'),\sqsubset)$ are isomorphic. We show that there is a unique largest principal ideal in $(\Rel^*(L),\sqsubset)$ and that it is isomorphic to $(L\setminus \{0 \},\leq)$. The result follows easily. 
    

    The two posets $\Trs(L)$ and $\Trs(L')$ are empty if and only if there is no non-trivial relation in $L$ and $L'$ and this is the case if and only if the lattices have only $1$ element. Otherwise, we denote by $0$ and $1$ the least and greatest elements of $L$ and we assume that $0\neq 1$.

    The principal (order) ideal of $(\Rel^*(L),\sqsubset)$ generated by $(0,1)$ consists of all the elements $(x,y)$ such that there is a pullback diagram
    \[
    \xymatrix{
    x \pb \ar[r] \ar[d] & y \ar[d] \\
    0 \ar[r] & 1.
    }
    \]
    This condition implies that $x = 0$ and $0\neq y\leq 1$. Hence this ideal contains all the pairs $(0,y)$ for $y\in L$ such that $0\neq y$. Moreover for two such pairs we have $(0,x)\sqsubset (0,y)$ if and only if there is a pullback diagram
    \[
    \xymatrix{
    0 \pb \ar[r] \ar[d] & x \ar[d] \\
    0 \ar[r]& y,
    }
    \]
    that is if and only if $x\leq y$. In other words the principal ideal generated by $(0,1)$ is isomorphic to $(L\setminus \{0\},\leq)$.

    If $(x,y) \in \Rel^*(L)$, we denote by ${<}(x,y) {>}$ the principal ideal generated by $(x,y)$. Concretely, we have
    \[
    {<}(x,y) {>} = \{ (a,b) \in \Rel^{*}(L)\ |\ (a,b) \sqsubset (x,y) \}.
    \]
    Clearly if ${<}(x,y) {>}$ has maximal size, then $(x,y)$ is a maximal element of $(\Rel^*(L),\sqsubset)$. 
    
    Let $(x,y)$ be a maximal element in $(\Rel^*(L),\sqsubset)$. Since $x<y \leq 1$, we have a pullback diagram
    \[
    \xymatrix{
    x \pb \ar[r] \ar[d] & y\ar[d] \\
    x \ar[r] & 1,
    }
    \]
    so by maximality, we have $(x,y) = (x,1)$. 

    The ideal $I = {<}(x,1){>}$ is the set of \emp{non-trivial relations} $(x\land b,b)$ such that $b\leq 1$. We have $x\land b = b$ if and only if $b\leq x$. Hence
    \[
    I = \{ (x\land b,b)\ |\ b\leq 1 \hbox{ and } b\not\leq x\}.
    \]
    Hence $I$ has cardinality $|L|-|(x)|$ where $(x)$ is the principal ideal in $(L,\leq)$ generated by $x$. Since $L$ has $0$ as least element, we have $|(x)| = \left\{\begin{array}{cc}1 & \hbox{ if $x=0$,}  \\\geq 2 &\hbox{ otherwise.} \end{array}\right.$ 
    
    In particular the order ideal generated by a non-trivial relation $(x,y)$ has cardinality $|L|-1$ if and only if $(x,y) = (0,1)$, and it has a strictly smaller cardinality otherwise.
    
\end{proof}

\section{Semidistributivity and trimness}
\subsection{Distributivity}
The notion of \emp{semidistributive} lattice is a weakening of the one of \emp{distributive} lattice. Hence before proving that the lattice of transfer systems is always semidistributive, it is natural to ask under what conditions  the lattice of transfer systems is distributive. Unfortunately for a finite lattice this is almost never the case.
\begin{proposition}
Let $L$ be a finite lattice. Then $\Trs(L)$ is \emp{distributive} if and only if $|L|\leq 2$.
\end{proposition}
\begin{proof}
If $|L| = 1$, then there is only one transfer system on $L$ which is $\Delta(L)$. If $|L|=2$, since $L$ is a lattice it has a greatest element and a least element, so it is a total order with $2$ elements. There are two transfer systems $\Delta(L)$ and $L$, and the lattice of transfer systems is also a total order with $2$ elements. Both of these lattices are distributive. 

Conversely, if $|L|\geq 3$ since $L$ has a greatest element and a least element, there is a chain (a set of pairwise comparable elements) of length at least $2$ in $L$ starting at the least element. Without loss of generality we assume that it is labelled by $0\to 1\to 2\to \cdots$, where $0$ is the least element of $L$ and $(0,1), (1,2)$ are cover relations in $L$. Our objective is to construct from this chain a pentagonal sublattice of $\Trs(L)$ and conclude that $\Trs(L)$ is not distributive.

We consider $\Tr(0,2)$ the smallest transfer system containing $(0,2)$. By \cref{lem:smallest_ts}, we have:
\[
\Tr(0,2) = \{ (0\land c,c) | c\leq 2\} \cup \Delta(L) = \{ (0,x) \ |\ x\leq 2\}\cup \Delta(L). 
\]
It is join-irreducible and $\Tr(0,2)_{*} = \Tr(0,2)_{(0,2)} =  \{ (0,x) \ |\ x< 2\}\cup \Delta(L)$.
Similarly, we have 
\[ 
\Tr(1,2) = \{ \{ (1\land c,c) | c\leq 2\} \cup \Delta(L)\}  = \{ (1,2) \} \cup \{ (0,c)\ |\ c\leq 2, 1\not\leq c\} \cup \Delta(L)  
\]
Since $(0,2) \notin \Tr(1,2)$ we have $\Tr(1,2) \cap \Tr(0,2) = \Tr(1,2) \cap \Tr(0,2)_*$. Moreover $\Tr(1,2) \lor \Tr(0,2) = \big(\Tr(1,2) \cup \Tr(0,2)\big)^{tc} = \Tr(0,2) \cup \{(1,2)\}$. Similarly, $\Tr(1,2) \lor \Tr(0,2)_* = \big(\Tr(1,2) \cup \Tr(0,2)_*\big)^{tc} = \{ (1,2) \} \cup \Tr(0,2)_* \cup \{(0,2)\} = \Tr(0,2) \lor \{(1,2)\}$. Here $(0,2) =(0,1) \cdot (1,2)$ with $(0,1) \in \Tr(0,2)_*$ and $(1,2) \in \Tr(1,2)$. We have obtained the following sublattice
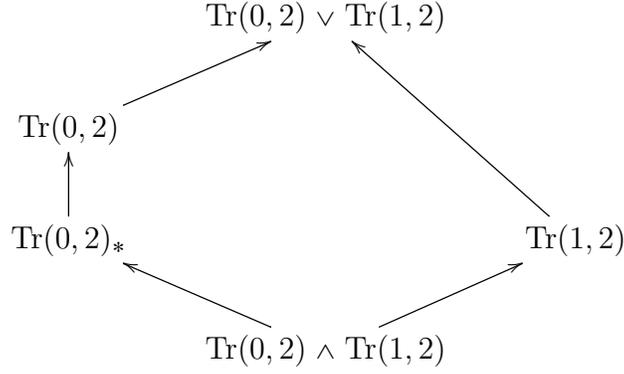
\[
\xymatrix{
& \Tr(0,2) \lor \Tr(1,2) \\
\Tr(0,2)\ar[ur] & \\
\Tr(0,2)_* \ar[u]&  & \Tr(1,2)\ar[luu] \\
& \Tr(0,2)\land \Tr(1,2)\ar[ru]\ar[lu]
}
\] 
This proves that $\Trs(L)$ is not a distributive lattice. 
\end{proof}
\begin{remark}
More generally, it is natural to ask if there is an easy characterization of the categories $\C$ such that the lattice of weak factorization systems is distributive. This is the case of the category of sets by the classification of \cite{Antolin}. 
\end{remark}
\subsection{Semidistributivity}
\begin{theorem}
Let $L$ be a finite lattice. The lattice of transfer systems on $L$ is \emp{semidistributive}.
\end{theorem}
\begin{proof}
The first step is to see that it suffices to prove that the lattice of transfer systems on $L$ is join-semidistributive. Recall from \cref{pro:anti-iso}, that $\Rc \mapsto \,^{\boxslash} \Rc$ is an isomorphism from the lattice of transfer systems on $L$ to the lattice of left saturated sets on $L$ (ordered by reverse inclusion). Moreover, by \cref{lem:duality} taking the opposite of a left saturated set induces an anti-isomorphism from the lattice of left saturated sets on $L$ to the lattice of transfer systems on $L^{op}$. Hence there is a lattice anti-isomorphism say $\Phi$ from $\Trs(L)$ to $\Trs(L^{op})$. If $\Trs(L^{op})$ is join-semidistributive, it follows that $\Tr(L)$ is meet-semidsitributive. 


Now we prove the join-semidistributivity. Let $\Rc_1,\Rc_2,\Rc_3$ be three transfer systems on $L$ such that $\Rc_1 \lor \Rc_2 = \Rc_1 \lor \Rc_3$. Note that $\Rc_1 \lor (\Rc_2 \land \Rc_3) \leq \Rc_1 \lor \Rc_2$ holds in any lattice. Conversely, we consider $(x,y) \in P = \Rc_1 \lor \Rc_2$. We assume that $(x, y)$ is a cover relation in the poset $P$. By \cref{pro:wfs_lattice}, the poset $P$ is the transitive closure of $\Rc_1 \cup \Rc_2$. Hence, the cover relation $(x,y)$ is either in $\Rc_1$ or in $\Rc_2$. Since $P = \Rc_1 \lor \Rc_3 = (\Rc_1 \cup \Rc_3)^{tc}$, the relation $(x,y)$ is also in $\Rc_1$ or in $\Rc_3$. It follows that $(x,y)$ is either in $\Rc_1$ or in $\Rc_2 \cap \Rc_3$ and we have:
\[
\Rc_1\lor \Rc_2 \subseteq \big(\Rc_1 \cup (\Rc_2 \cap \Rc_3)\big)^{tc} = \Rc_1 \lor \big(\Rc_2\land \Rc_3\big). 
\]
\end{proof}
\subsection{Trimness}\label{sec:trim}
We recall that an element $x$ of a lattice $L$ is \defn{left modular} if for any $y<z$ in $L$, we have
\[
(y\lor x) \land z = y \lor (x\land z).
\]
A chain in a lattice is called \defn{maximal}, when it is maximal with respect to inclusion.

\begin{definition}
A lattice $(L,\land,\lor)$ is \defn{trim} if  it has a maximal chain of $n+1$ left modular elements, exactly $n$ join-irreducibles and exactly $n$ meet-irreducibles. 
\end{definition}
Trim lattices were introduced by Hugh Thomas in \cite{thomas_trim} as ungraded analogue to distributive lattices. In particular a graded trim lattice is distributive. The typical examples of trim lattices are the \emp{cambrian lattices} associated to Dynkin diagrams (see \cite{thomas_trim,muhle2016trimness}). 

Let us recall that the \defn{length} of a lattice is the maximum of the length of its chains. A lattice of length $n$ has at least $n$ join-irreducibles and $n$ meet-irreducibles.
\begin{definition}
 A lattice $(L,\leq)$ is \defn{extremal} if its length is equal to the number of join-irreducible elements and to the number of meet-irreducible elements.
 \end{definition}
 \begin{proposition}\label{pro:extremal}
 Let $(L,\leq)$ be a finite lattice. Then the lattice of transfer systems on $L$ is \emp{extremal}.
 \end{proposition}
 \begin{proof}
 By \cref{pro:join-irreducible}, the lattice $\Trs(L)$ has $|\Rel^{*}(L)|$ join-irreducible elements. Since there is an anti-isomorphism between $\Trs(L)$ and $\Trs(L^{op})$ and $L^{op}$ has as many relations as $L$, this is also the number of meet-irreducible elements. If $\Delta(L) \leq \Rc_1 \leq \Rc_2 \leq \cdots \leq \Rc_n \leq L$ is a chain, each inclusion is strict so at each step there is a relation in $\Rc_{i}$ which is not in $\Rc_{i-1}$. Hence the length of such a chain is bounded by the total number of non-trivial relations on $L$. 
 
 Conversely, we need to construct a chain of length exactly $|\Rel^*(L)|$. Recall from \cref{lem:cov_relations} that in order to construct a cover relation $\Rc_1 \lessdot \Rc$ in $\Trs(L)$, we need to choose a good cover relation in $\Rc^{c}$. We consider $(x,y) \in \Rc^c$ with $y$ is maximal in $(L,\leq)$. That is if there is another relation $(a,b)$ in $\Rc^c$ such that $y \leq b$, then $y=b$. Suppose that there is a pullback diagram as follow, where $(u,v) \in \Rc$.
 \[
 \xymatrix{
 x\pb \ar[r]\ar[d] & y \ar[d]\\
 u \ar[r]_{\in \Rc} & v
 }
 \]
 As $(x, y)$ is non-trivial, so is $(u,v)$. Then $(u,v)$ can be written as $u \to e \to v$ with $(u,e)\in \Rc$ and $(e,v) \in \Rc^c$. By maximality of $y$ we have $v = y$ and hence $x= y \land u = v\land u = u$. That is $(u,v) = (x,y)$. Hence the relation $(x,y)$ is maximal among all relations in $\Rc$ with respect to $\sqsubset$. So by \cref{lem:lower_ts}, $\Rc_{(x,y)} = \Rc \setminus \{ (x,y) \}$.  In any transfer system $\Rc\neq \Delta(L)$ we can always find such a relation $(x,y)$, so we can always construct a lower cover relation where we remove exactly $1$ relation. So we can construct a maximal chain of length $|\Rel^*(L)|$ starting from $L$ and ending at $\Delta(L)$. 
 \end{proof}
 \begin{theorem}
 Let $(L,\leq)$ be a finite lattice. Then the lattice of transfer systems on $L$ is a \emp{trim} lattice.
 \end{theorem}
 \begin{proof}
 By \cite[Theorem 1.4]{TW} an extremal semidistributive lattice is trim. Hence the result follows from \cref{thm:semidistributive} and \cref{pro:extremal}.
 \end{proof}

  \section{Consequences of semidistributivity}
 In this section, we investigate some of the consequences of the semidistributivity of the lattice of weak factorization systems. We start by a discussion on the canonical join representations for such lattices, then we discuss the existence of a labelling of the Hasse diagram by the join-irreducible elements. In a semidistributive lattice there is always a canonical bijection between the set of join-irreducibles and the meet-irreducibles which is called $\kappa$. We describe this map for the lattice of weak factorization systems. 
 \subsection{Canonical join representations and join-labelling}\label{sec:canonical_join}
 Let $(L,\leq)$ be a lattice. A subset $B\subseteq L$ is called a \defn{join representation} of $a\in L$ if $a = \lor B$. We say that $B$ is \defn{irredundant} if for all proper subset $C$ of $B$ we have $\lor C < \lor B$. If a set $B$ is an irredundant join representation, it is an \emp{antichain} in $L$ ( the elements are pairwise incomparable). We say that $C$ is a \defn{refinement} of $B$, if for every $c\in C$ there exists $b \in B$ such that $c \leq b$. This is a preorder on the subsets of $L$ which restricts to a partial order on the antichains of $L$. 

\begin{definition}
A join representation $a = \lor B$ is \defn{canonical} if it is irredundant and $\forall C \subseteq L$ such that $a = \lor C$, we have that $B$ is a refinement of $C$. 
\end{definition}
In other words a canonical join representation is the \emp{least antichain} (if it exists) with respect to the refinement order among antichains joining to $a$. The existence of such join-representation is heavily related to semidistributivity, as we have:
\begin{theorem}
Let $L$ be a finite lattice. The following are equivalent:
\begin{enumerate}
\item $L$ is a join-semidistributive lattice.
\item Every element of $L$ has a canonical join representation.
\end{enumerate}
\end{theorem}
\begin{proof}
See for example \cite[Theorem 2.24]{freese}. 
\end{proof}
Of course there is the dual notion of \defn{canonical meet representation} and it is related to \emp{meet-semidistributivity}. 

Let $L$ be a join-semidistributive lattice and $x\lessdot y$ be a cover relation in $L$. It is known (see for example \cite{thomas2021introduction} ) that there is a \emp{unique join-irreducible} $j\in L$ such that 
\[
x \lor j = y \hbox{ and } x\lor j_* = x.
\]
The join-irreducible $j$ is called the \defn{join label} of the cover relation $x\lessdot y$. 

\begin{lemma} \label{lem:join label}
Let $(L,\leq)$ be a finite lattice. Then the \emp{join label} of a cover relation $\Rc_1 \lessdot \Rc$ is $\Tr(a,b)$ where $(a,b)$ is the unique non trivial relation in $\Rc \cap \, ^{\boxslash} \Rc_1$. 
\end{lemma}
\begin{proof}
By \cref{lem:cov_relations}, there is a unique non-trivial relation $(a,b) \in \Rc \cap \, ^{\boxslash} \Rc_1$. Hence $\Tr(a,b) \subseteq \Rc$ and we have \[ \Rc_1 \subseteq \Rc_1 \lor \Tr(a,b) \subseteq \Rc.\] Since $\Rc_1 \subseteq \Rc$ is a cover relation, we have $\Rc_1 \lor \Tr(a,b)  = \Rc_1$ or $\Rc_1 \lor \Tr(a,b)  =\Rc$. As $(a,b) \notin \Rc_1$, we have $\Rc_1 \lor \Tr(a,b) = \Rc$. Let us recall that $\Tr(a,b)_* = \Tr(a,b)\setminus \{a,b\}$ and we have similarly,
\[
\Rc_1 \subseteq \Rc_1 \lor \Tr(a,b)_* \subseteq \Rc. 
\]
If $\Rc_1 \lor \Tr(a,b)_*  = \Rc$ we have $(a,b) \in \Rc_1 \lor \Tr(a,b)_*$. By the second item of \cref{lem:cov_relations} the relation $(a,b)$ is a cover relation in $\Rc$, hence we have $(a,b) \in \Rc_1$ or $(a,b) \in \Tr(a,b)_*$. This is a contradiction in both cases, hence we have $\Rc_1 =  \Rc_1 \lor \Tr(a,b)_*$.  Since $\Tr(a,b)$ is join-irreducible by \cref{pro:join-irreducible}, it is the join label of $\Rc_1 \lessdot \Rc$.
\end{proof}
Of course, we will simply label it as $(a,b)$ instead of $\Tr(a,b)$, with the cover relation $\Rc_1 \lessdot \Rc$.  

By duality if $L$ is a meet-semidistributive lattice and $x\lessdot y$ is a cover relation in $L$, there is a \emp{unique meet-irreducible} $m \in M$ such that 
\[
y \land m = x \hbox{ and } y \land m^* = y.
\]
This meet-irreducible element is called the \defn{meet label} of $x\lessdot y$. 
When the lattice is semidistributive, it has the two labels. Hence we can define the \defn{map $\kappa$} from the set of join-irreducible elements to the set of meet-irreducible elements of $L$ by setting $\kappa(j)$ as the meet label of the cover relation $j_* \lessdot j$. Dually we define the \defn{map $\kappa^d$} by setting for a $m$ meet-irreducible element $\kappa^d(m)$ as the join label of the cover relation $m \lessdot m^*$.

\begin{proposition}
Let $L$ be a finite semidistributive lattice. Then $\kappa$ and $\kappa^d$ are two inverse bijections between the set of join-irreducible elements and the set of meet-irreducible elements of $L$. 
\end{proposition}
\begin{proof}
See for example \cite[Proposition 9.2]{thomas2021introduction}. 
\end{proof}
\subsection{Elevating graphs and elevating sets}
By \cref{pro:anti-iso} and \cref{lem:duality} the \emp{meet-irreducible transfer systems} of $L$ are of the form $\Ls(a,b)^{\boxslash}$. Indeed, for all $a<b$ in $L$, $(\Ls(a,b))^{op}= \Tr(b,a)$ is join-irreducible in the lattice of transfer systems on $L^{op}$, hence $\Ls(a,b)$ is meet-irreducible in the lattice of left saturated sets of $L$, and $\Ls(a,b)^* = \Ls(a,b)\setminus \{(a,b)\}$ (remember that the order of $\Lss(L)$ is the reverse inclusion). Using the isomorphic map $-^{\boxslash}$, it follows that $\Ls(a,b)^{\boxslash}$ is meet-irreducible in the lattice of transfer systems of $L$, and $(\Ls(a,b)^{\boxslash})^* =(\Ls(a,b)^*)^{\boxslash}$. 

\begin{proposition}\label{pro:kappa}
Let $(L,\leq)$ be a finite lattice and $(a,b) \in \Rel^*(L)$. Then \[ \kappa(\Tr(a,b)) = \Ls(a,b)^{\boxslash} = (a,b)^{\boxslash}.\] 
\end{proposition}
\begin{proof}
First we prove that $\Ls(a,b)^{\boxslash} = (a,b)^{\boxslash}$. Since $(a,b) \in \Ls(a,b)$ we clearly have $\Ls(a,b)^{\boxslash} \subseteq (a,b)^{\boxslash}$. Conversely let $(x,y) \in (a,b)^{\boxslash}$. The elements of $\Ls(a,b)$ are obtained by pushout diargams like
\[
\xymatrix{
a \po \ar[r] \ar[d] & c \ar[d]\\
b \ar[r] & b\lor c.
}
\]
Hence if there is a commutative square involving $(c,b\lor c)$ and $(x,y)$ we can complete it as
\[
\xymatrix{
a \po \ar[r] \ar[d] & c \ar[r] \ar[d] & x \ar[d] \\
b \ar[r] & b\lor c \ar[r] & y.
}
\]
Since $(x,y) \in (a,b)^{\boxslash}$, we have $b\leq x$, hence we have $b\lor c \leq x$, so $(x,y) \in \Ls(a,b) ^{\boxslash}$. By duality, we have $\, ^{\boxslash}\Tr(a,b) = \, ^{\boxslash}(a,b)$. 

It follows that $\Tr(a,b) \cap \Ls(a,b)^{\boxslash} = \Tr(a,b) \cap (a,b) ^{\boxslash} \subseteq \Tr(a,b)$. By \cref{lem:cov_relations} we have $\Tr(a,b) \cap \, ^{\boxslash} \Tr(a,b)_* = \{(a,b)\} \cup \Delta(L)$. Hence $(a,b) \in \,^{\boxslash}\Tr(a,b)_*$, so we have $\Tr(a,b)_* \subseteq (a,b)^{\boxslash}$. In other words we have:
\[
\Tr(a,b)_* \subseteq \Tr(a,b) \cap (a,b)^{\boxslash} \subseteq \Tr(a,b).
\]
Since $(a,b) \notin  \Tr(a,b) \cap (a,b)^{\boxslash}$ and $\Tr(a,b)_* \subseteq \Tr(a,b)$ is a cover relation, we have \[ \Tr(a,b)_* = \Tr(a,b) \cap (a,b)^{\boxslash}.\] 
It remains to prove that $\Tr(a,b) \subseteq (\Ls(a,b)^{*})^{\boxslash}$. This is equivalent to proving $\Ls(a,b)^* \subseteq \,^{\boxslash}(a,b)$. Let $(x,y) \in \Ls(a,b)$ and consider a commutative diagram
\[
\xymatrix{
x \ar[r] \ar[d] & a\ar[d] \\
y \ar[r] & b
}
\]
Since $(x,y) \in \Ls(a,b)$, we can complete it as
\[
\xymatrix{
a \po \ar[r] \ar[d] & x \ar[r] \ar[d] & a\ar[d] \\
b \ar[r] & y \ar[r] & b,}
\]
so $(x,y) = (a,b)$. Hence if $(x,y) \in \Ls(a,b)^*$ there is no commutative diagram involving $(x,y)$ and $(a,b)$, so we have $(x,y) \boxslash (a,b)$. 
\end{proof}
\begin{corollary}
  Let $(L,\leq)$ be a finite lattice. Then the \emp{meet label} of a cover relation $\Rc_1 \lessdot \Rc$ is $(a,b)^{\boxslash}$ where $(a,b)$ is the unique non trivial relation in $\Rc \cap \, ^{\boxslash} \Rc_1$.   
\end{corollary}
\begin{proof}
    Since $\Tr(a,b)$ is the join label of $\Rc_1 \lessdot \Rc$ by \cref{lem:join label}, it follows from \cite[Proposition 2.13]{barnard2022exceptional} that the meet-label of $\Rc_1\lessdot \Rc$ is $\kappa(\Tr(a,b))= (a,b)^{\boxslash}$.
\end{proof}

The map $\kappa$ is important for us because of the following result. 

\begin{theorem}[Barnard-Hanson]\label{thm:barnard_hanson}
Let $(L,\leq)$ be a finite semidistributive lattice. There is a bijection between the elements of $L$ and the subsets $S$ of the join-irreducible elements of $L$ such that for every $x,y \in S$ such that $x\neq y$ we have $x \leq \kappa(y)$.
\end{theorem}
\begin{proof}
Since $L$ is a semidistributive lattice, every element $x \in L$ has a canonical join recomposition $x = \lor A$. If $x = \lor A = \lor B$ for two sets of join-irreducible elements, then $A$ is smaller than $B$ for the refinement order. In particular if $A$ and $B$ are two canonical join representation of the same element they are equal. It follows that there is a bijection between the elements of $L$ and the set of all the canonical join representations of the elements of $L$. By \cite[Theorem C]{barnard2022exceptional}, a subset $S$ of join-irreducible elements is a canonical join representation if and only if it satisfies the condition of the theorem. 
\end{proof}
\begin{definition}
Let $(L,\leq)$ be a finite lattice. Then a subset $S\subseteq \Rel^*(L)$ is an \defn{elevating set} if $\forall (a,b) \neq (c,d) \in S$ we have $(a,b) \boxslash (c,d)$. 
\end{definition}
We are now ready to prove \cref{thm:semibricks}. 
\begin{proof}[Proof of \cref{thm:semibricks}]
This is a reformulation of \cref{thm:barnard_hanson}  for the lattice of transfer systems of $L$ where we use \cref{pro:kappa} and \cref{pro:join-irreducible}. Then join-irreducibles are $\Tr(a,b)$ for $(a,b) \in \Rel^*(L)$ and the condition $\Tr(a,b) \subseteq \kappa(\Tr(c,d))$ becomes $\Tr(a,b) \subseteq (c,d)^{\boxslash}$ which is equivalent to $(a,b) \in (c,d)^{\boxslash}$ and equivalent to $(c,d) \boxslash (a,b)$. 
\end{proof}
The set of canonical join representations is a simplicial complex and it has been proved by Barnard in \cite[Theorem 1.2]{barnard_canonical}  that it is a \emp{flag complex}. This means that this is the clique complex of its $1$-skeleton. For transfer systems, using \cref{pro:kappa} and \cref{pro:join-irreducible} this $1$-skeleton is isomorphic to the following graph:
\begin{definition}
Let $(L,\leq)$ be a finite lattice. The \defn{elevating graph} of $L$ is the graph where the vertices are $\Rel^*(L)$ the non-trivial relations of $L$. There is an edge between $(a,b)$ and $(c,d)$ if and only if $(a,b) \boxslash (c,d)$ and $(c,d) \boxslash (a,b)$. 
\end{definition}

The proof of \cref{thm:cliques} follows. Indeed, there are as many transfer systems as there are canonical join representations. The canonical complex is the clique complex of the elevating graph, so there are as many elements in this complex as there are cliques in the elevating graph. We refer to \cref{sec:examples} for illustrations in various examples.  

 \section{Congruences}

Recall that an equivalence $\equiv$ on a lattice $L$ is called a \defn{congruence} if it has the property: for any elements $x_1, x_2, y_1$ and $y_2$ in $L$ such that $x_1 \equiv y_1$ and $x_2 \equiv y_2$, then also $x_1 \wedge x_2 \equiv y_1 \wedge y_2$ and $x_1 \vee x_2 \equiv y_1 \vee y_2$. We denote by $\Con(L)$ the set of congruences on $L$, partially ordered by the refinement ($\theta_1 \leq \theta_2$ if and only if $\forall x,y \in L$, $x \equiv_{\theta_1} y$ implies $x\equiv_{\theta_2} y$). It is well known that the lattice $\Con(L)$ is distributive \cite{funayama}.  It follows from the Birkhoff theorem that $\Con(L) = \ideal(P)$ where $P$ is the poset of \emp{join-irreducible congruences}. 

Given a congruence $\equiv$ on $\Con(L)$, we say that $\equiv$ contracts an arrow $q: x\rightarrow y$ ($x\lessdot y$) of the Hasse diagram of $L$ whenever $x\equiv y$. We define $\con(q)$ to be the smallest congruence relation contracting $q$. That is, $\con(q)$ is the meet of all congruences contracting $q$. By \cite[Proposition 9.5.14]{gratzerlattice}, the map sending a join-irreducible element $j\in L$ to $\con(j_* \to j)$ induces a surjective map from the set of join-irreducible elements of $L$ to the set of join-irreducible congruences of $L$. Similarly, the map sending a meet irreducible element $m$ to $\con(m\to m^*)$ induces a surjective map from the set of meet-irreducible elements of $L$ to the set of meet-irreducible congruences of $L$. 

\begin{definition}
A finite lattice $(L,\leq)$ is \defn{congruence uniform} if the maps $j \mapsto \con(j_*\to j)$ and $m\mapsto \con(m\to m^*)$ are bijective. 
\end{definition}
By \cite{day} a congruence uniform lattice is always semidistributive, and conversely knowing that a lattice is semdisitrbutive can be used to simplify the proof of the congruence uniformity. 
\begin{lemma}\label{lem:reduction}
Let $(L,\leq)$ be a finite lattice. If $L$ is semidistributive and the map $j \mapsto \con(j_* \to j)$ from the set of join-irreducible elements of $L$ to the set of join-irreducible congruences is injective, then $L$ is congruence uniform. 
\end{lemma}
\begin{proof}
See the paragraph 7.2 of \cite{FTSDL}. 
\end{proof}
For two arrows $q$ and $q'$ in the Hasse diagram of $L$, we say that $q$ and $q'$ are \defn{forcing equivalent} if $\con(q) = \con(q')$. Moreover, there is a preorder on the set of all arrows in the Hasse diagram, called the \defn{forcing preorder}. It is defined by 
\[
q \rightsquigarrow q' \leftrightarrow \con(q) \geq \con(q') \hbox{ in the lattice}  \Con(L).
\]

\subsection{Congruence uniformity}\label{sec:congruence_uniform}
 
Let $(L, \leq)$ be a finite lattice and $\Rc_1 \subseteq \Rc$ be two transfer systems on $\Trs(L)$. We denote by $\Rel^{*}(\Rc \cap ^{\boxslash}\Rc_{1})$ and $\jirr[\Rc_1,\Rc]$ the set of non-trivial relations in $\Rc \cap ^{\boxslash}\Rc_{1}$ and the set of elements which are join-irreducible in the interval $[\Rc_1, \Rc]$, respectively. Then we have the following: 
\begin{lemma}\label{lem:bij_interval_rel}
There is a bijection $\Psi: \jirr[\Rc_1,\Rc] \rightarrow \Rel^{*}(\Rc \cap \, ^{\boxslash}\Rc_1)$ mapping $\Rc^{'}\in \jirr[\Rc_1. \Rc]$ to $(a, b)$, where $(a,b)$ is the join label of $\Rc_{*}^{'} \lessdot \Rc^{'}$. Moreover, we have $\Rc^{'}= \Rc_1 \vee \Tr(a,b)$ and $\Rc_{*}^{'}= \Rc' \wedge (a,b)^{\boxslash}$.
\end{lemma}

\begin{proof}
Let $\Rc^{'}$ be a join-irreducible transfer system in the interval $[\Rc_1, \Rc]$. Then $\Rc^{'}$ covers a unique transfer system $\Rc'_*$ in this interval. By \cref{lem:join label}, the join-label $(a,b)$ of this cover relation is the unique non-trivial relation in $\Rc' \cap \, ^{\boxslash}\Rc'_{*}$. Since $\Rc' \subseteq \Rc$ and $\, ^{\boxslash}\Rc'_{*} \subseteq \,^{\boxslash}\Rc_1$, we have $(a,b) \in \Rc \cap \, ^{\boxslash}\Rc_{1}$ and the map $\Psi$ is well defined.

Next, we prove that the map $\Psi$ is injective. If there is another join-irreducible transfer system $\Rc^{''} \in [\Rc_1, \Rc]$ such that $\Psi(\Rc^{''}) = (a,b)$. Denote by $\Rc_{*}^{''}$ the unique transfer system such that $\Rc_{*}^{''} \lessdot \Rc^{''}$. Then $(a,b)$ is also the unique non-trivial relation in $\Rc^{''} \cap \, ^{\boxslash}\Rc_{*}^{''}$. Hence, $(a,b) \in \Rc' \cap \Rc''$. Moreover, the transfer system $\Rc'\cap \Rc''$ is in the interval $[\Rc_1, \Rc]$ and it satisfies $\Rc' \cap \Rc'' \subseteq \Rc'$. Since $\Rc'$ is join-irreducible in this interval, we have $\Rc' \cap \Rc'' = \Rc'$ or $\Rc' \cap \Rc'' \subseteq \Rc'_*$. Since $(a,b) \notin \Rc'_*$, we have $\Rc' \cap \Rc'' = \Rc'$, so $\Rc'\subseteq \Rc''$. By symmetry, we also obtain the other inclusion, and $\Rc'= \Rc''$.

It remains to prove that $\Psi$ is surjective. Let $(a,b)\in \Rel^{*}(\Rc \cap \, ^{\boxslash}\Rc_1)$ and $\Rc^{'}= \Rc_1 \vee \Tr(a,b)$.  As $(a,b)\in \Rc$ and $\Rc_1 \subseteq \Rc$, it is immediate that $\Rc_1 \subseteq \Rc^{'} \subseteq \Rc$. Let also $\Rc_{*}^{'}= \Rc^{'} \wedge (a,b)^{\boxslash}$. As $(a,b) \in \, ^{\boxslash}\Rc_{1}$, equivalently, $\Rc_{1} \subseteq (a,b)^{\boxslash}$. So $\Rc_1 \subseteq \Rc_*' \subseteq \Rc$. Moreover, by construction we have $\Rc'_* \subseteq \Rc'$ and since $(a,b) \in \Rc'$ and $(a,b)\notin \Rc'_*$, we have 

\[ \Rc_1 \subseteq \Rc_{*}^{'} \subsetneq \Rc^{'}. \]  

Let $X$ be a transfer system such that $\Rc_1 \subseteq X \subsetneq \Rc^{'}$. We will prove that $X\subseteq \Rc'_*$. Since $X \subseteq \Rc_{*}^{'}$ if and only if $X \subseteq (a,b)^{\boxslash}$, to prove $X\subseteq \Rc_{*}^{'}$, we only need to prove $X \subseteq (a,b)^{\boxslash}$. Note that $(a,b)\notin X$, otherwise we get a contradiction that $\Rc^{'}= \Rc_1 \vee \Tr(a,b) \subseteq \Rc_1 \vee X = X$. 

Let $(x,y)\in X$. In particular $(x,y) \in \Rc_1 \lor \Tr(a,b)$, so the relation $(x,y)$ is obtained by transitivity from relations in $\Rc_1$ and $\Tr(a,b)$. Note that $\Rc_1 \subseteq (a,b)^{\boxslash}$ and $\Tr(a,b)_* \subseteq (a,b)^{\boxslash}$. Hence, $\Rc_1 \lor \Tr(a,b)_* \subseteq (a,b)^{\boxslash}$. So if all the relations appearing in the decomposition of $(x,y)$ are in $\Rc_1$ or in $\Tr(a,b)_*$, we obtain that $(x,y) \in (a,b)^\boxslash$. The only relation of $\Tr(a,b)$ which is not in $\Tr(a,b)_*$ is $(a,b)$, so let us assume that it appears in $(x,y)$, that is we have $(x,y) = x\to a \to b \to y$. If $x = a$, then by the pullback property of $X$, we have $(a,b) \in X$ and this is a contradiction. So we have $x < a$, and there is no commutative diagram
\[
\xymatrix{
a \ar[r] \ar[d] & x\ar[d] \\
b \ar[r]& y.
}
\]
In particular we have $(x,y) \in (a,b)^{\boxslash}$ and we have $X\subseteq \Rc'_*$. This proves that $\Rc'$ is join-irreducible in the interval $[ \Rc_1, \Rc]$ and that we have a cover relation $\Rc'_* \lessdot \Rc'$. It remains to compute its join label. To do so, we recall that $X\mapsto \,^{\boxslash}X$ is an isomorphism between the lattice of transfer systems and the lattice of left saturated sets (ordered by reverse inclusion). The inverse isomorphism is $Y \mapsto Y^{\boxslash}$. So, we have 

\[ \,^{\boxslash} \Rc'_* = \,^{\boxslash}\big(\Rc'\cap (a,b)^{\boxslash}\big) = \,^{\boxslash}\Rc' \land \Ls(a,b).\]
Since the lattice of left saturated set is ordered by reverse inclusion, we have 
\[ (a,b) \in \Ls(a,b)\subseteq \,^{\boxslash}\Rc' \land \Ls(a,b) = \,^{\boxslash} \Rc'_*.\] 
Hence, we have
\[ (a,b) \in \Rc' \cap \,^{\boxslash} \Rc'_*.  \]
So, the relation $(a,b)$ is the unique non-trivial relation in this intersection, and it is the join-labelling of the cover relation $\Rc'_* \lessdot \Rc'$. This concludes the proof of the surjectivity of the map $\Psi$.

 
 \end{proof}

Every non-trivial relations $a\lneq b$ on $L$, can be seen as an interval $\{ x\in L\ |\ a\leq x \leq b\}$ of the poset $L$, hence they can be naturally ordered by \emp{containment}:
\[ 
(a,b) \subseteq (c,d) \Longleftrightarrow c\leq a \hbox{ and } b \leq d. 
\] 
In the following we associate to each $(a,b) \in \Rel^{*}(L)$ a certain congruence in $\Con(\Trs(L))$. We define $\equiv_{(a,b)}$ as follow: For $\Rc_1, \Rc_2 \in \Trs(L)$, we put $\Rc_1 \equiv_{(a,b)} \Rc_2$ if and only if $(a, b) \subseteq (x,y)$ for every non-trivial relation $(x,y) \in (\Rc_1 \vee \Rc_2) \cap \, ^{\boxslash}(\Rc_1 \wedge \Rc_2)$.

\begin{proposition}
  The relation $\equiv_{(a,b)}$ is a congruence relation for $\Trs(L)$.
\end{proposition}

\begin{proof}
The symmetry of $\equiv_{(a,b)}$ is clear, and the reflexivity follows from the first item of \cref{lem:step1}. Let $\Rc_1 \equiv_{(a,b)} \Rc_2$ and $\Rc_2 \equiv_{(a,b)} \Rc_3$. Firstly, we need to prove $\Rc_1 \equiv_{(a,b)} \Rc_3$ to ensure that $\equiv_{(a,b)}$ is an equivalence. For every non-trivial relation 
\[
(x,y)\in (\Rc_1 \vee \Rc_3) \cap \, ^{\boxslash}(\Rc_1 \wedge \Rc_3),
\]
since $\Rc_2$ is a transfer system, by \cref{pro:defts2} the pair $(\,^{\boxslash} \Rc_2,\Rc_2)$ is a weak factorization system, then $(x,y)$ can be factorized as follows:

\[
\xymatrix{ 
x\ar[rd]_{^{\boxslash}\Rc_2 \, \ni }\ar[rr]& &y\\
& \exists z \ar[ur]_{\in \Rc_2}
}
\]

Note that $x\neq z$ or $z\neq y$. If $x \neq z$, since $\Rc_1 \vee \Rc_3$ is a transfer system and $(x,z)$ is a pullback of $(x,y) \in \Rc_1 \vee \Rc_3$, then we have 
\[
(x,z) \in (\Rc_1 \vee \Rc_3) \cap \, ^{\boxslash}\Rc_2.
\]
It follows from $(x,z) \in \Rc_1 \vee \Rc_3$ that $(x,z)$ has a factorization $x \rightarrow z_0 \rightarrow z$ where $(z_0, z)$ is a cover relation in $\Rc_{1}^{c}$ or $\Rc_{3}^{c}$. On the other hand, we can see from the following pushout diagram 
 \[
 \xymatrix{
 x\po \ar[r]  \ar[d]  & z \ar@{=}[d] \\
  z_0 \ar[r] & z
 }
 \]

that 
\[(z_0, z) \in \Rc_1 \cap \, ^{\boxslash}\Rc_2 \subseteq (\Rc_1 \vee \Rc_2) \cap \, ^{\boxslash}(\Rc_1 \wedge \Rc_2),\] or 
\[(z_0, z) \in \Rc_3 \cap \, ^{\boxslash}\Rc_2 \subseteq (\Rc_2 \vee \Rc_3) \cap \, ^{\boxslash}(\Rc_2 \wedge \Rc_3).
\]
Note that $\Rc_1 \equiv_{(a,b)} \Rc_2$ and $\Rc_2 \equiv_{(a,b)} \Rc_3$.
So we obtain $(a,b) \subseteq (z_0, z) \subseteq (x,y)$. If $z\neq y$, the proof is dual. So, we proved $\equiv_{(a,b)}$ is an equivalence and it remains to check that it is compatible with the meet and the join of the lattice. 

Let $\Rc_1, \Rc_2, \Tc_1, \Tc_2$ be transfer systems such that $\Rc_1 \equiv_{(a,b)} \Tc_1$ and  $\Rc_2 \equiv_{(a,b)} \Tc_2$. 
Denote by $\Rc^v = \Rc_1 \vee \Rc_2$ and $\Tc^v = \Tc_1 \vee \Tc_2$. For every non-trivial relation $(u,v)\in (\Rc^v \vee \Tc^v) \cap \, ^{\boxslash}(\Rc^v \wedge \Tc^v)$. It follows from 
\[(u,v) \in \Rc^v \vee \Tc^v = (\Rc_1 \vee \Tc_1) \vee (\Rc_2 \vee \Tc_2)
\]
that $(u,v)$ can be factored as a composition of $\Rc_1 \vee \Tc_1$ and $\Rc_2 \vee \Tc_2$. So we suppose that $(u, v)$ has a factorization $u \rightarrow v_0 \rightarrow v$, where $(v_0, v)$ is non-trivial and $(v_0, v) \in \Rc_1 \vee \Tc_1$, or $(v_0, v) \in \Rc_2 \vee \Tc_2$. Moreover, it follows from $(v_0,v)$ is a pushout of $(u,v) \in \, ^{\boxslash}(\Rc^v \wedge \Tc^v)$ that $(v_0,v) \in \, ^{\boxslash}(\Rc^v \wedge \Tc^v)$. 

On the other hand, since 
\[ (\Rc_1 \wedge \Tc_1) \vee (\Rc_2 \wedge \Tc_2) \subseteq \Rc^v \wedge \Tc^v,
\]
we have
\[
\, ^{\boxslash}(\Rc^v \wedge \Tc^v) \subseteq \, ^{\boxslash}(\Rc_1 \wedge \Tc_1) \quad 
 \text{and} \quad \, ^{\boxslash}(\Rc^v \wedge \Tc^v) \subseteq \, ^{\boxslash}(\Rc_2 \wedge \Tc_2).   
 \]

Then we have $(v_0, v) \in (\Rc_1 \vee \Tc_1) \cap \, ^{\boxslash}(\Rc_1 \wedge \Tc_1)$, or $(v_0, v) \in (\Rc_2 \vee \Tc_2) \cap \, ^{\boxslash}(\Rc_2 \wedge \Tc_2)$. It follows $(a,b) \subseteq (v_0, v) \subseteq (u, v)$. So we proved $\Rc^v \equiv_{(a,b)} \Tc^v$.

Let $\Rc_m = \Rc_1 \wedge \Rc_2$ and $\Tc_m = \Tc_1 \wedge \Tc_2$. For any non-trivial relation 
\[(u,v) \in (\Rc_m \vee \Tc_m) \cap \, ^{\boxslash}(\Rc_m \wedge \Tc_m),\] 
we have
\[ (u,v) \in \Rc_m \vee \Tc_m = (\Rc_1 \wedge \Rc_2) \vee (\Tc_1 \wedge \Tc_2) \subseteq \Rc_1 \vee \Tc_1.
\]
Note that 
\[
(u,v)\in \,^{\boxslash}(\Rc_m \wedge \Tc_m)= \, ^{\boxslash}(\Rc_1 \wedge \Tc_1 \wedge \Rc_2 \wedge \Tc_2) = \, ^{\boxslash}(\Rc_1 \wedge \Tc_1) \wedge \, ^{\boxslash}(\Rc_2 \wedge \Tc_2).
\]
Here the last meet is computed in the lattice of left saturated sets ordered by reverse inclusion. 

Without loss of generality, we can assume that $(u,v)$ has a composition $u \rightarrow u_1 \rightarrow v$, where $(u,u_1)$ is non-trivial and $(u,u_1) \in \, ^{\boxslash}(\Rc_1 \wedge \Tc_1)$.   Then we proved 
\[(a,b) \subseteq (u, u_1) \subseteq (u, v),
\]
which means that $\Rc_m \equiv_{(a,b)} \Tc_m$. 
\end{proof}

 Then we have the following:
\begin{proposition}\label{prop:forcing_equiv}
The arrows $q$ and $q'$ are forcing equivalent if and only they have the same join label. 
\end{proposition}

\begin{proof}
First, we prove that if two arrows have the same label, they are forcing equivalent. This is not particular to our situation, as it holds for any semidistributive lattice. If $j$ is the join label of $q : x\to y$ and $q' : x'\to y'$. Then by \cite[Proposition 2.13]{barnard2022exceptional}, we have $x\land j = j_*$ and since $j\leq y$, we have $y \land j = j$.

let $\equiv$ be a congruence relation contracting $q$. Then we have 
\[
j = y \land j \equiv x \land j = j_*.
\]
And 
\[
y' = x'\lor j \equiv x'\lor j_* = x'.
\]
So $\equiv$ also contracts the arrow $q'$ and by symmetry we see that $q$ and $q'$ are forcing equivalent. 




On the other hand, assume that the arrows $q : \Tc \to \Rc$ and $q' : \Tc' \to \Rc'$ are forcing equivalent. We denote by $(a,b)$ and $(a',b')$ their join labels, respectively. The congruence $\equiv_{(a,b)}$ defined above contracts $q$, so it contracts $q'$. Then we have $\Tc' \equiv_{(a,b)} \Rc'$. Hence it follows $(a,b) \subseteq (a',b')$. Conversely, we also have $(a', b') \subseteq (a, b)$, so we have $(a,b)=(a',b')$. 
\end{proof}
We are ready to prove \cref{thm:cong_unif}.
\begin{proof}[Proof of \cref{thm:cong_unif}.]
By \cref{thm:semidistributive}, the lattice of transfer systems on $L$ is semidistributive, hence by \cref{lem:reduction}, it is enough to prove that $j \mapsto \con(j_* \to j)$ is injective. Since the join-label of $j_* \to j$ is equal to $j$, this follows from \cref{prop:forcing_equiv}. 
\end{proof}

\subsection{Lattice of congruences}
Let $(a,b)$ be a non-trivial relation in $\Rel^*(L)$. We denote by $\Con\big((a,b)\big)$ the smallest congruence contracting $\Tr(a,b)$ and $\Tr(a,b)_{*}$. We have the following: 

\begin{lemma}\label{lemma:Con(q) and Con(q') 1}
Let $a< b <c$ be elements in $L$. We denote by $q=(a,c)$ and $q'=(b,c)$ the non-trivial relations of $\Rel^*(L)$. Then  $\Con(q) \subseteq \Con(q')$.  
\end{lemma}

\begin{proof}
It is clear that $\Tr(a,c) \leq \Tr(a,b) \vee \Tr(b,c)$ since $(a,c)$ is the composition of $(a,b)$ and $(b,c)$. From $\Tr(b,c) \equiv_{\Con(q')} \Tr(b,c)_{*}$, then it follows that 
\[
\Rc:= (\Tr(a,b) \vee \Tr(b,c)_{*}) \wedge \Tr(a,c) \equiv_{\Con(q')} (\Tr(a,b) \vee \Tr(b,c)) \wedge \Tr(a,c) = \Tr(a,c). 
\]
Since $(a,c) \notin \Tr(a,b) \vee \Tr(b,c)_{*}$,  we have  $\Rc \lneq \Tr(a,c)$, that is, $\Rc \leq \Tr(a,c)_{*}$. Then we have $\Tr(a,c) = \Tr(a,c) \lor \Tr(a,c)_* \equiv_{\Con(q')} \Rc \lor \Tr(a,c)_{*} = \Tr(a,c)_*$. Note that $\Con(q)$ is the smallest congruence contracting $\Tr(a,c)$ and $\Tr(a,c)_{*}$. We have $\Con(q) \subseteq \Con(q')$.

\end{proof}

\begin{lemma} \label{lemma: Con(q) and Con(q') 2}
Let $a< b <c$ be elements in $L$. We denote by $q=(a,c)$ and $q'=(a,b)$ the non-trivial relations of $\Rel^*(L)$. Then  $\Con(q) \subseteq \Con(q')$.  
\end{lemma}

\begin{proof}
Similar to \cref{lemma:Con(q) and Con(q') 1}, we immediately have 
\[
\Rc:= (\Tr(a,b)_{*} \vee \Tr(b,c)) \wedge \Tr(a,c) \equiv_{\Con(q')} (\Tr(a,b) \vee \Tr(b,c)) \wedge \Tr(a,c) = \Tr(a,c). 
\]
Since $\Rc \leq \Tr(a,c)_{*}$ and $\Tr(a,c) \equiv_{\Con(q')} \Tr(a,c)_{*}$. It also follows $\Con(q) \subseteq \Con(q')$.

\end{proof}

Then we have the following proposition. 
\begin{proposition}\label{prop:forcing}
Let $L$ be a finite lattice. Denote by $q=(a,b)$ and $q'=(c,d)$ be two non-trivial relations in $\Rel^*(L)$. Then $(a,b) \subseteq (c,d)$ if and only if $\Con(q') \subseteq \Con(q)$.    
\end{proposition}

\begin{proof}
Let  $q = (a,b) \subseteq (c,d) = q'$. It immediately follows from  \cref{lemma:Con(q) and Con(q') 1} and \cref{lemma: Con(q) and Con(q') 2} that $\Con(q') \subseteq \Con(q)$.

Suppose $\Con(q') \subseteq \Con(q)$. Due to $\Tr(c,d) \equiv_{\Con(q')} \Tr(c,d)_{*}$ and 
\[
\Con(q') \subseteq \Con(q) \subseteq \,\equiv_{q},\]
we have $\Tr(c,d) \equiv_{q} \Tr(c,d)_{*}$. 
As $(c,d) \in \Tr(c,d) \cap \, ^{\boxslash}\Tr(c,d)_{*}$, it follows $(a,b) \subseteq (c,d)$.
\end{proof}
\begin{proof}[Proof of \cref{thm:lat_of_cong}.]
Let $(L,\leq)$ be a finite lattice. Then the lattice of congruences on the lattice of transfer systems of $L$ is distributive, hence it is isomorphic to $\ideal(P)$ where $P$ is the poset of join-irreducible congruences. Since the lattice of transfer systems is congruence uniform with join-irreducibles in bijection with the non-trivial relations of $L$, the set $P$ is in bijection with the set $\Rel^*(L)$. By \cref{prop:forcing}, this poset is isomorphic to the opposite of $(\Rel^*(L),\subseteq)$. 
\end{proof}
\section{The spine of the lattice of weak factorization systems}

 A trim lattice has an important distributive sublattice called the \defn{spine}. It consists of the union of all the chains of maximal length. By \cite[Theorem 3.7]{TW}, the spine is also the set of the \emp{left-modular elements} of the lattice and by \cite[Lemma 6]{thomas_trim} it is a \emp{distributive sublattice}.

\begin{lemma} \label{lemma: to prove spin}
Let $(L,\leq)$ be a finite lattice and $(\Lc,\Rc)$ be a weak factorization system on $L$. Let $(a,b)\in \Rel^{*}(L)$. Then $(a, b) \in \Lc$ if and only if there is a chain 
\[
\Rc =\Rc_0 \lessdot \Rc_1 \cdots \lessdot \Rc_n =L
\]
such that $(a,b)$ is the join label of $\Rc_j \lessdot \Rc_{j+1}$ for some $0 \leq j \leq n-1$.
\end{lemma}

\begin{proof}
Let $(\Lc, \Rc)$ be a weak factorization system and $(a,b)$ a non-trivial relation in $\Lc= \, ^{\boxslash}\Rc$. Then it follows from \cref{lem:bij_interval_rel} that there is a cover relation $\Rc_j \lessdot \Rc_{j+1}$ in the inverval $[\Rc, L]$ whose join-label is $(a,b)$. Since $\Trs(L)$ is finite and $\Rc_j \lessdot \Rc_{j+1}$ is in the interval of $[\Rc, L]$. Then the chain above exists.

On the other hand, suppose such a chain exists. Since $(a,b)$ is the join label of $\Rc_j \lessdot \Rc_{j+1}$ for some $j$, then it is clear that $(a,b) \in \Rc_{j+1}\cap \, ^{\boxslash} \Rc_j$. Certainly we have 
\[(a,b) \in \, ^{\boxslash}\Rc_{j} \subseteq \, ^{\boxslash}\Rc = \Lc.\]
\end{proof}

\begin{lemma}\label{lem:spine}
    $(L,\leq)$ be a finite lattice and $(\Lc,\Rc)$ be a weak factorization system on $L$. Then
    \begin{enumerate}
        \item The non trivial relations of $\Lc$ are the join labels of the cover relations in the interval $[\Rc,L]$.
        \item The non trivial relations of $\Rc$ are the join labels of the cover relations in the interval $[\Delta(L),\Rc]$. 
    \end{enumerate}
\end{lemma}
\begin{proof}
Since $(\Lc,\Rc)$ is a weak factorization system, we have $\Lc = \,^{\boxslash} \Rc$. Hence $\Lc = \,^{\boxslash}R \cap L$. Similarly, we have $\Rc = \Rc \cap \,^{\boxslash}\Delta(L)$. The result follows from \cref{lem:bij_interval_rel}. 
\end{proof}
 \begin{proposition}
Let $(L,\leq)$ be a finite lattice. Then a weak factorization system $(\Lc, \Rc)$ is on the spine of the lattice $\Wfs(L)$ if and only if $\Lc \cup \Rc = \Rel(L)$. 
 \end{proposition}

\begin{proof}
We recall that the join label of two distinct cover relations in the same chain are distinct. Indeed if we have $x\lessdot y \leq z\lessdot t$ and $j$ denotes the join label of $x\lessdot y$, then $j\leq y$ so $z\lor j = z$ and $j$ is not the join label of the cover relation $z\lessdot t$. 

If  \[\Rc_0 =\Delta(L) \lessdot \Rc_1 \lessdot \Rc_2 \lessdot \cdots \lessdot \Rc_{n-1} \lessdot L = \Rc_{n}\] is a chain of maximal length, then all the non-trivial relations of $L$ appear as join label of the cover relations of this chain. So if $(\Lc,\Rc)$ is a weak factorization system on the spine, then by \cref{lem:spine}, $\Rc$ is one of the $\Rc_i$ and the join labels of $[\Delta(L),\Rc_i]$ are the non-trivial relations of $\Rc$ and the join-labels of $[\Rc_i,L]$ are the non-trivial relations of $\Lc$. So, all the non-trivial relations of $L$ are either in $\Lc$ or in $\Rc$ and we have $\Rc \cup \Lc = \Rel(L)$.


On the other hand let $(\Lc,\Rc)$ be a weak factorization system such that $\Lc \cup \Rc = \Rel(L)$. Since a non-trivial relation cannot be in $\Lc$ and in $\Rc$ at the same time, this condition implies that $\Rel^*(L) = \Rel^*(\Lc) \sqcup \Rel^*(\Rc)$. By the proof of \cref{pro:extremal}, there is a chain of length $|\Rel^*(\Rc)|$ from $\Delta(L)$ to $\Rc$. Dually in the poset of left saturated set, there is a chain of length $|\Rel^*(\Lc)|$ from $\Lc$ to $\Delta(L)$. Applying the isomorphism, between left saturated sets and transfer systems, we obtain a chain of length $|\Rel^*(\Lc)|$ from $\Rc$ to $L$. Hence, a chain of length $|\Rel^*(L)|$ from $\Delta(L)$ to $L$ containing $\Rc$. So $(\Lc,\Rc)$ is on the spine. 


\end{proof}
 
\section{Galois graph of the lattice of weak factorization systems} 
For \emp{distributive lattices}, we have the Birkhoff representation theorem which says that any finite distributive lattice can be represented as subsets of its set of join-irreducible elements. There is a similar, rather technical, result for finite semidistributive lattices \cite{FTSDL}, involving \emp{pairs of subsets} of the set of join-irreducible elements with a maximality property. For finite \emp{extremal lattices}, we have the \defn{Markowsky representation theorem} \cite{markowsky1992primes}, which also represents the lattice as pairs of vertices of a \emp{graph} on the set of join-irreducible elements. 

Let $G$ be a simple directed graph on $[n]=\{1,2,\cdots,n\}$ such that $i\to k$ implies $k < i$. For $X \subseteq [n]$ and $Y \subseteq [n]$ such that $X\cap Y = \emptyset$, we say that $(X,Y)$ is an \emp{orthogonal pair} if there are no arrows from $x\in X$ to $y\in Y$. We say that $(X,Y)$ is a \emp{maximal orthogonal pair} if $X$ and $Y$ are maximal for that property. We denote by $\pairs(G)$ the set of all maximal orthogonal pairs of $G$. This set can be viewed as a poset for the ordering 
\[
(X,Y) \preceq (X',Y') \hbox{ if } X\subseteq X' \hbox{ and } Y' \subseteq Y.
\]
The poset $(\pairs(G),\preceq)$ is a lattice where the join is computed by taking the intersection of the second component, and the meet is given by taking the intersection of the first component. By Markowsky's representation theorem \cite{markowsky1992primes}, this lattice is \emp{extremal} and conversely every extremal lattice is isomorphic to the lattice of maximal orthogonal pairs of a direct graph associated to $L$ called the \defn{Galois graph} of $L$. See also \cite[Theorem 2.4]{TW}. In our context, this result boils down to the following. 

\begin{theorem}
Let $(L,\leq)$ be a finite lattice. The lattice of weak factorization systems on $L$ is isomorphic to the lattice of maximal orthogonal pairs of the graph $G_L$ whose vertices are the non-trivial relations of $L$ and there is an arrow $(a,b) \to (c,d)$ if and only $c \leq a$, $d\leq b$ and $d\not\leq a$ i.e., the morphism $c\to d$ does not lift on the left the morphism $a\to b$. 
\end{theorem}
\begin{proof}
By Markowsky's representation theorem, it is enough to show that the Galois graph of the lattice of transfer systems on $L$ is isomorphic to the graph $G_L$. Since the lattice of transfer systems is both extremal and congruence uniform, we can use the description of the Galois graph given in \cite[Corollary A.1]{muhle2021noncrossing}. The Galois graph of $\Trs(L)$ is isomorphic to the graph whose vertices are the join-irreducible elements of the lattice $\Trs(L)$. Moreover, there is an arrow $j_s \to j_t$ if and only if $j_t \leq (j_t)_{*} \lor j_s$. 

By \cref{pro:join-irreducible}, the join-irreducible transfer systems are the $\Tr(a,b)$ for $(a,b) \in \Rel^*(L)$. Since $\Tr(c,d)$ is the smallest transfer system containing $(c,d)$, when $R$ is a transfer system we have $\Tr(c,d) \subseteq R$ if and only if $(c,d)\in R$. So $\Tr(c,d) \subseteq \Tr(c,d)_* \lor \Tr(a,b)$ if and only if $(c,d) \in  \Tr(c,d)_* \lor \Tr(a,b)$. By \cref{lem:smallest_ts}, we have $\Tr(c,d)_* \lor \Tr(a,b) = \big(\Tr(c,d)_* \cup \Tr(a,b) \big)^{tc}$. Moreover $\Tr(c,d)_* = \Tr(c,d) \setminus \{ (c,d) \}$. 

$\Tr(c,d)$ is a transfer system and $(c,d) \in \Tr(c,d)$, so for every $x\in L$ such that $c\leq x \leq d$, we have $(c,x) \in \Tr(c,d)$. Moreover if $x<d$ we have $(c,x) \in \Tr(c,d)_*$. So if $(c,d) \in  \big(\Tr(c,d)_* \cup \Tr(a,b) \big)^{tc}$, there are $c_1,c_2,\cdots, c_{n-1}$ such that $c=c_0 < c_1 < c_2 < \cdots < c_{n-1} < c_n=d$ such that $(c_i,c_{i+1}) \in \Tr(c,d)_*$ or $(c_i,c_{i+1}) \in \Tr(a,b)$.  By the remark above $(c,c_{n-1}) \in \Tr(c,d)_*$, and $(c_{n-1},d) \in \Tr(a,b)$, otherwise we would have $(c,d) \in \Tr(c,d)_*$. Hence we obtain a factorization $c\to d = c\to x \to d$ where $c\leq x$, $(c,x) \in \Tr(c,d)_*$, $(x,d) \in \Tr(a,b)$ and $x\neq d$. In particular, we have a pullback diagram:
\[
\xymatrix{
x \pb \ar[r] \ar[d] & d \ar[d]\\
a \ar[r] & b.
}
\]
So we have $c\leq x \leq a$, $d \leq b$ and $x = a\land d$. Since $x\neq d$ the last condition implies that $d\not\leq a$. So the morphism $c\to d$ does not lift on the left $a\to b$. 

Conversely, if $c\to d$ does not lift on the left $a\to b$, we have $c\leq a$, $d\leq b$ and $d\not\leq a$. We deduce the diagram
\[
\xymatrix{
c\ar[r] & a\land d \pb \ar[r] \ar[d] & d\ar[d]\\
& a \ar[r] & b.
}
\]
Since $d\not\leq a$, we have $a \land d < d$. As above, $(c,d) \in \Tr(c,d)$ and $\Tr(c,d)$ is a transfer system so $(c,a\land d) \in \Tr(c,d)$. And since $a\land d \neq d$ we have $(c, a\land d) \in \Tr(c,d)_*$. Moreover, by construction $(a\land d,d) \in \Tr(a,b)$ hence $(c,d) \in \Tr(c,d)_* \lor \Tr(a,b)$. 
\end{proof}

We deduce another characterization of the \emp{spine} of the lattice of transfer systems. By \cite{thomas_trim}, the spine is a distributive lattice so it is isomorphic to $\ideal(P)$ where $P$ is its poset of join-irreducible elements. By \cite[Proposition 2.6]{TW}, the poset $P$ is obtained by taking the transitive closure of Galois graph of the trim lattice. In our setting we recover the second natural way of ordering the intervals of a poset. For two relations, we set $(a,b)\preceq (c,d)$ if there is a commutative diagram
\[
\xymatrix{
c \ar[r]\ar[d] & a \ar[d]\\
d \ar[r] & b.
}
\]
Clearly, this is equivalent to $c\leq a$ and $d\leq b$. Hence, we recover the product order on the intervals (of size at least two) of $L$.
\begin{proposition}
Let $(L,\leq)$ be a finite lattice. The \emp{spine} of the lattice of transfer systems on $L$ is isomorphic to $\ideal(\Rel^*(L),\preceq)$.
\end{proposition}
\begin{proof}
If there is a path from $(a,b)$ to $(c,d)$ in the Galois graph of the lattice of transfer systems, then $c\leq a$ and $d\leq b$. 

Conversely, if $(a,b)$ and $(c,d)$ are two non-trivial relations such that $c\leq a$ and $d\leq b$, then we have $c < d \leq b$, so $(c,b)$ is a non-trivial relation in $L$. Moreover we have a commutative diagram:
\[
\xymatrix{
c \ar[r]\ar[d] & c \ar[d]\ar[r]& a\ar[d] \\
d \ar[r] & b \ar[r] & b.
}
\]
Since $(c,d)$ is a non-trivial relation, we have $d\not\leq c$ and similarly $b\not\leq a$. If $b\neq d$ and $c \neq a$, we have a path $(a,b) \to (c,b) \to (c,d)$ in the Galois graph of $L$. If $b = d$, then the situation is even simpler: we have a commutative diagram
\[
\xymatrix{
c \ar[r]\ar[d]& a \ar[d]\\
b \ar[r] & b,
}
\]
and $b \not\leq a$ since $(a,b)$ is a non-trivial relation, so there is an arrow $(a,b) \to (c,b)$ in the Galois graph of $L$. Similarly if $c=a$, there is an arrow $(a,b)\to (a,d)$ in the Galois graph of $L$. 
\end{proof}

\section{Lattice of $G$-transfer systems}\label{sec:gts}

Let $G$ be a finite group. Recall that a \defn{$G$-transfer system} is a transfer system on the lattice $\operatorname{Sub}(G)$ of the subgroups of $G$ which is \emp{stable by conjugation}. We denote by $\Tr(G)$ the poset of $G$-transfer systems ordered by inclusion of their relations. By \cite[Proposition 3.7]{franchere2021self}, this is a lattice where the meet is given by the intersection. The join of two $G$-transfer systems is the smallest $G$-transfer system containing the union. The objective of this short section is to explain how the general properties of the lattice of weak factorization systems also apply to the lattice of the $G$-transfer systems. 

\begin{lemma}\label{lem:sublattice}
Let $G$ be a finite group. Then the lattice of $G$-transfer systems is a sublattice of the lattice of transfer systems on $\operatorname{Sub}(G)$.
\end{lemma}
\begin{proof}
Since the meets in both lattices are the intersection, we only have to check that if $\Rc$ and $\Rc'$ are two $G$-transfer systems, the joins of $\Rc$ and $\Rc'$ in the two lattices are equal.

By \cite[Theorem A.2]{rubin}, the smallest $G$-transfer system containing a set of relations $S$ is obtained by closing $S$ under conjugation, then pullback and then by transitivity. As in the proof of \cref{lem:smallest_ts}, we see that the join of $\Rc$ and $\Rc'$ in $\Tr(G)$ is given by $(\Rc \cup \Rc')^{tc}$. This is also the join of $\Rc$ and $\Rc'$ in the lattice of transfer systems on $\operatorname{Sub}(G)$.
\end{proof}

Note that a \emp{sublattice} of a trim lattice is \emp{not a trim lattice} in general, hence for trimness we need another argument. 

Let us recall that a poset $(X,\leq)$ is called a \defn{$G$-poset} if there is a monotone $G$-action on it. That is $X$ is a $G$-set and for every $x\leq y \in X$, and every $g\in G$ we have $g\cdot x\leq g\cdot y$. 

A lattice $(L,\leq)$ is a \defn{$G$-lattice} if it is a $G$-poset and the action is compatible with the meets and the joins of the lattice. That is $g\cdot (x\land y) = g\cdot x \land g\cdot y$ and $g\cdot (x\lor y) = g\cdot x \lor g\cdot y$, for every $x,y \in L$ and $g\in G$. 

We start by an easy lemma.
\begin{lemma}
Let $(L,\leq)$ be a lattice which is also a $G$-poset. Then $L$ is a $G$-lattice.
\end{lemma}
\begin{proof}
Let $x,y\in L$. Then $x\leq x\lor y$ so $g\cdot x \leq g\cdot (x\lor y)$. Similarly $g\cdot y \leq g\cdot (x\lor y)$, hence $g\cdot x\lor g\cdot y \leq g\cdot (x\lor y)$. Moreover we have $g\cdot x \leq g\cdot x \lor g\cdot y$, hence $x\leq g^{-1}\cdot \big( g\cdot x \lor g\cdot y\big)$, and similarly for $y$. So $x\lor y \leq g^{-1}\cdot \big( g\cdot x \lor g\cdot y\big)$, and we get $g\cdot (x\lor y) \leq g\cdot x \lor g\cdot y$. The proof for the meet is dual. 
\end{proof}
If $(L,\leq)$ is a $G$-lattice and $\Rc$ is a transfer system on $L$, we denote by $g\cdot \Rc$ the relation on $L$ defined by 
\[
x (g\cdot \Rc) y \Longleftrightarrow (g^{-1}\cdot x) \Rc (g^{-1}\cdot y).
\]
We use $g^{-1}$ instead of $g$ in order to have an action of $G$ on the left. 
\begin{lemma}\label{lem:gts}
Let $(L,\leq)$ be a $G$-lattice, then the lattice of transfer systems on $L$ is a $G$-lattice for the action defined above.
\end{lemma}
\begin{proof}
Let $R$ be a transfer system on $L$. It is clear that $g\cdot R$ is reflexive and transitive. Moreover if $x (g\cdot R) y$, we have $(g^{-1}\cdot x) R (g^{-1}\cdot y)$, so $g^{-1}\cdot x \leq g^{-1}\cdot y$. Multiplying by $g$, we get $x\leq y$. Hence $g\cdot R$ is a subposet of $(L,\leq)$. 

Now we consider a pullback diagram in $(L\leq)$:
\[
\xymatrix{
x\land z \pb \ar[r]\ar[d] & z\ar[d] \\
x \ar[r]_{\in g\cdot \Rc} & y
}
\]
such that $x (g\cdot \Rc) y$. We have $g^{-1}\cdot z \leq g^{-1}\cdot y$, and a pullback diagram
\[
\xymatrix{
(g^{-1}\cdot x) \land (g^{-1}\cdot z) \pb \ar[r]\ar[d] & g^{-1}\cdot z \ar[d]\\
g^{-1}\cdot x \ar[r]_{\in \Rc} & g^{-1}\cdot y
 }
\]
Since $\Rc$ is a transfer system and $g^{-1}\cdot (x\land z) = (g^{-1}\cdot x) \land (g^{-1}\cdot z)$, we get $g^{-1}\cdot (x\land z) \Rc g^{-1}\cdot z$, so $(x\land z) (g\cdot \Rc) z$. 

It remains to see that the action is an action of $G$-poset, since this is straightforward we leave the details to the reader. 
\end{proof}

When $(L,\leq)$ is a $G$-lattice, we denote by $L^{G} = \{ l\in L\ |\ g\cdot l = l\ \forall g\in G\}$ the sublattice of $L$ consisting of the \defn{$G$-fixed points}. 

\begin{lemma}\label{lem:gfp}
Let $G$ be a finite group and $L = \operatorname{Sub}(G)$ its lattice of subgroups. Then $L$ is a $G$-lattice, so $\Trs(L)$ is a $G$-lattice and $\Trs(L)^{G} = \Trs(G)$. 
\end{lemma}
\begin{proof}
It is clear that the lattice of subgroups of $G$ is a $G$-lattice for the action of $G$ by conjugation, so by \cref{lem:gts} the lattice of transfer systems on $L$ is a $G$-lattice. 

Let $\Rc\in \Trs(L)^{G}$, $g\in G$, and $H,K\in L$. Then we have 
\begin{align*}
H \lhd_\Rc K & \Longrightarrow H \lhd_{g^{-1}\cdot \Rc} K \hbox{ since $g^{-1}\cdot \Rc=\Rc$} \\
& \Longrightarrow g\cdot H \lhd_\Rc g\cdot K,
\end{align*}
so $\Rc$ is a $G$-transfer system. 

Conversely let $\Rc$ be a $G$-transfer system and $g\in G$. For $H,K\in L$ such that $H\lhd_\Rc K$. Then we have $g^{-1}\cdot H \lhd_\Rc g^{-1}\cdot K$ since $\Rc$ is a $G$-transfer system, hence $\Rc \subseteq g\cdot \Rc$. Moreover if $H\lhd_{g\cdot \Rc} K$, we have $g^{-1}\cdot H \lhd_\Rc g^{-1}\cdot K$. Since $\Rc$ is a $G$-transfer system, we have $g \cdot g^{-1}\cdot H \lhd_ \Rc g\cdot g^{-1}\cdot K$, so $H\lhd_\Rc K$ and $g\cdot \Rc \subseteq \Rc$. 
\end{proof}
\begin{proof}[Proof of \cref{thm:gts}]
It is clear that a sublattice of a semidistributive lattice is semidistributive. A sublattice of a congruence uniform lattice is also congruence uniform by \cite[Theorem 4.3]{day}. So by \cref{lem:sublattice}, the lattice of $G$-transfer systems is congruence uniform and semidsitributive. By \cite[Theorem 4]{thomas_trim} and \cref{lem:gfp}, the lattice of $G$-transfer systems is also a trim lattice. 
\end{proof}
\begin{remark}
The lattice of $\mathcal{S}_3$-transfer systems can be found in \cite[Figure 4]{rubin}, however one arrow of the Hasse diagram is not at the correct position. 
\end{remark}

\section{Examples}\label{sec:examples}
\subsection{Case of a total order}\label{sec:total_order}
There are no new results in this section, however we would like to point out that the notion of weak factorization systems on a total order can be easily related to the \emp{classical poset point of view on the binary trees}. This gives an alternative proof of \cite[Theorem 25]{roitzheim2022n}. 

\emp{Binary search trees} are binary trees with a particular labelling of their vertices. They were introduced in computer science in the sixties and it is known since \cite{BW,hivert2005algebra} that they are central objects for the combinatorics of the Tamari lattice viewed as a quotient of the lattice of weak ordering of the permutations. 

Let $T$ be a \emp{binary tree} with $n$ inner vertices, we label its vertices by the integers $1,2,\cdots,n$ in such a way that if $x$ is a vertex in the left (right) subtree of $y$ then the label if $x$ is strictly smaller (larger) than the label of $y$. There is a unique way of viewing a binary tree as binary search tree which can be obtained by an in-order traversal of the tree (recursively visit left subtree, root and right subtree). We refer to \cref{fig:bin-search-tree} for an illustration. 

\begin{figure}[h]
\centering
\scalebox{.7}{\begin{tikzpicture}
  [ level distance=8mm,
   level 1/.style={sibling distance=15mm},
   level 2/.style={sibling distance=10mm},
   level 3/.style={sibling distance=5mm},
   inner/.style={circle,draw=black,fill=black!15,inner sep=0pt,minimum size=4mm,line width=1pt},
   leaf/.style={},                      
   edge from parent/.style={draw,line width=1pt}]
  \node [inner] {4}
     child {node [inner] {2}
       child {node [inner] {1}
         child {node [leaf] {}}
         child {node [leaf] {}}
       }
       child {node [inner] {3}
         child {node [leaf] {}}
         child {node [leaf] {}}
       }
     }
     child {node [inner] {5}
       child {node [leaf] {}
		}
       child {node [leaf] {}}
     };
\end{tikzpicture}}
\caption{Binary search tree of size 5 labeled by the in-order algorithm.}
\label{fig:bin-search-tree}
\end{figure}
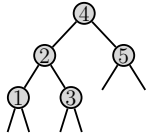

If $T$ is such a binary search tree, it induces a partial order $\lhd_T$ on $[n]$ by setting $i \lhd_T j$ if the vertex labelled by $i$ is in the subtree with root the vertex labelled by $j$. Conversely, it is not difficult to classify the posets coming from such a binary search tree.

\begin{proposition}\label{prop:charac_tree_poset}
A poset $\lhd$ on $[n]$ is of the form $\lhd_T$ for a binary tree $T$ if and only if for every $x,y\in [n]$ such that $x<y$ we have:
\begin{enumerate}
\item If $x\lhd y$ and $x\leq z\leq y$, then $z\lhd y$.
\item If $y\lhd x$ and $x\leq z\leq y$, then $z\lhd x$.
\item For every $x\leq y \in [n]$,there exists $z\in [n]$ such that $x\leq z \leq y$ and $x\lhd z$ and $y\lhd z$. 
\end{enumerate}
\end{proposition}
\begin{proof}
This is \cite[Proposition 2.21]{cpp}.
\end{proof}
Since $[n]$ is a total ordering for the usual ordering, we can split the relations of $\lhd_T$ into two sets $\Inc(T)$ and $\Dec(T)$. The first set is the set of \defn{increasing relations} of $\lhd_T$: the relations $x\lhd_T y$ for which $x\leq y$. The second set is the set of \defn{decreasing relations} of $\lhd_T$: the relations $y\lhd_T x$ for which $x\leq y$.

The first item of \cref{prop:charac_tree_poset} tells us that $\Inc(T)$ is a \emp{left-saturated set} for the total ordering $[n]$. Similarly, the second item tells us that $\Dec(T)$ is a \emp{left-saturated set} for the \emp{opposite total ordering} on $[n]$, so $\Dec(T)^{op}$ is a \emp{transfer system} on $[n]$. The third item is better rephrased using a diagram:
\[
\xymatrix{ 
x\ar[rd]_{\Inc(T) \ni }\ar[rr] & &y\ar[dl]^{\in \Dec(T)}\\
& \exists z 
}
\]
and we see that $(\Inc(T),\Dec(T)^{op})$ is a weak factorization system on $[n]$. Conversely, if $(\Lc,\Rc)$ is a weak factorization system on $[n]$, then $\Lc \cup \Rc^{op}$ is a poset satisfying the three conditions of \cref{prop:charac_tree_poset}, hence it comes from a binary tree.

\begin{theorem}
Let $L$ be the usual total ordering on $[n]$. The poset of weak factorization systems of $L$ is isomorphic to the Tamari lattice.
\end{theorem}
\begin{proof}
The map sending a binary tree $T$ to $(\Inc(T),\Dec(T)^{op})$ is a bijection between the set of binary trees with $n$ vertices and the weak factorization systems on $[n]$. It is well-known that for two binary trees $T_1$ and $T_2$ we have $T_1\leq T_2$ in the Tamari lattice if and only if $\Dec(T_1)\subseteq \Dec(T_2)$ and $\Inc(T_2) \subseteq \Inc(T_1)$ (see e.g. By \cite[Proposition 2.22]{cpp}). Hence this bijection is an isomorphism of posets. 
\end{proof} 
\begin{remark}
The canonical join-complex of the Tamari lattice has been studied in \cite{barnard2020canonical}. There is a description of the complex in Proposition 2.10. She proves in Theorem 1.1 that the complex is shellable and contractible or homotopy equivalent to a wedge of spheres. 
\end{remark}
\begin{remark}
Using this approach, it is easy to see that the \defn{composition closed} premodel structures in the sense of \cite{balchin2024composition} are exactly the \defn{exceptional intervals} of \cite{rognerud_exceptional}. These intervals are known to be in bijection with the intervals in the lattice of noncrossing partitions. This gives an alternative proof of \cite[Theorem 4.6]{balchin2024composition}. The exceptional intervals are related to the \emp{Serre functor} of the bounded derived category of the incidence algebra of the Tamari lattice (see \cite{rognerud_tamari} for more details). At the moment, we don't know if this is just a coincidence. 
\end{remark}
\subsection{Transfer systems for the group $C_p \times C_p$}
In this short section, we consider the lattice of subgroups of a product of two cyclic groups of order $p$. As a corollary, we obtain an alternative proof of \cite[Theorem 5.4]{bao2023transfer}.

These lattices of subgroups are the diamond lattices. Precisely, we let $Di_n$ be the lattice on ${0,1,\cdots, n+1}$ where $1,2,\dots,n$ is an antichain and $0$ is the least element of the lattice and $n+1$ the greatest. Its Hasse diagram is 
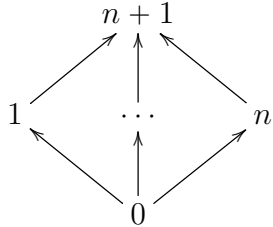
\[
\xymatrix{
&  n+1 &\\
1\ar[ru] & \cdots\ar[u] & n\ar[lu] \\ 
&  0\ar[ru]\ar[lu]\ar[u]
}
\]
\begin{proposition}
    The number of transfer systems for the lattice $Di_n$ is $2^{n+1}+n$.
\end{proposition}
\begin{proof}

The proof follows from the description of the \emp{elevating graph} of the lattice. The vertices are the non-trivial relations of the poset. They are of three types: $(0,x)$, $(x,n+1)$ for $x\in \{1,\cdots, n\}$ and $(0,n+1)$. There is an edge between $(a,b)$ and $(c,d)$ in the graph if and only if $(a,b)$ lifts on the left $(c,d)$ and vice-versa. 

If $x,y \in \{1,\cdots, n\}$ and $x\neq y$, there are no commutative diagrams involving $(0,x)$ and $(0,y)$ and also no commutative diagram involving $(x,n+1)$ and $(y,n+1)$. In particular they lift each other. So $\{ (0,x) \ |\ x\in \{1,2,\cdots, n\}\}$ is a complete graph $K_n$. Similarly for the relations $(x,n+1)$. If $(0,x)$ is a relation with $1\leq x \leq n$, we have
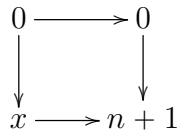
\[
\xymatrix{
0\ar[r]\ar[d] & 0 \ar[d]\\
x \ar[r]& n+1
}
\]
since $x\not\leq 0$, we see that $(0,x)$ does not lift on the left $(0,n+1)$. Similarly, we see that $(0,n+1)$ does not lift on the left $(x,n+1)$. So $(0,n+1)$ is an isolated vertex in the graph. 

It remains to see if there are edges between the two complete graphs. For $(0,x)$ and $(y,n+1)$ we have a commutative diagram
\[
\xymatrix{
0 \ar[r]\ar[d] & y \ar[d]\\
x \ar[r]& n+1
}
\]
and it splits into two triangles if and only if $x = y$. To conclude the elavating graph of $Di_n$ consists of 
\begin{enumerate}
\item A complete graph on the vertices $(0,x)$ for $x\in \{1,2,\cdots,n\}$,
\item A complete graph on the vertices $(x,n+1)$ for $x\in \{1,2,\cdots,n\}$
\item An isolated vertex $(0,n+1)$.
\item For every $x\in \{1,2,\cdots,n\}$, an edge between $(0,x)$ and $(x,n+1)$.
\end{enumerate}
\begin{figure}[h!]
\begin{center}
\[ \xymatrix{
 &  & \bullet \ar@{-}[rd] &  &  \\
 & \bullet \ar@{-}[ru] \ar@{-}[rr] &  & \bullet \ar@{-}[ld] &  \\
 & \bullet \ar@{-}[rd] \ar@{-}[ruu] & \bullet \ar@{-}[lu] \ar@{-}[uu] &  & \bullet \\
\bullet \ar@{-}[ru] \ar@{-}[rr] \ar@{-}[ruu] &  & \bullet \ar@{-}[ruu] &  &  \\
 & \bullet \ar@{-}[lu] \ar@{-}[ru] \ar@{-}[uu] \ar@{-}[ruu] &  &  & 
}
\]
\end{center}
\caption{The elevating graph of $Di_4$.}
\end{figure}
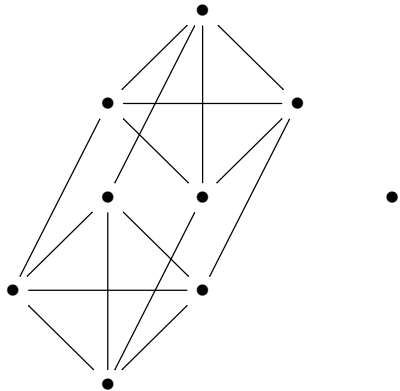

The number of cliques is:
\begin{enumerate}
\item $1$ empty subgraph
\item $1$ isolated vertex.
\item $n$ edges outside of the two complete graphs.
\item $2\times (2^{n}-1)$ non empty cliques in the two complete graphs. 
\end{enumerate}
It follows that the number of transfer systems for $Di_n$ is $1+1+n +2^{n+1} - 2 = 2^{n+1}+n$ 
\end{proof}

\subsection{A lower bound for the number of transfer systems of the boolean lattices}
Boolean lattices appear as lattices of subgroups of squarefree abelian groups. Hence they form an interesting family if one hopes to classify the $G$-transfer systems for all finite abelian groups $G$. The boolean lattices of subsets of $[n]$ for $n=0,1,2$ are respectively total orders and a diamond lattice. The case $n=3$ has been treated in \cite{balchin2020equivariant} and is already complicated. With our approach counting transfer systems is equivalent to counting cliques in the elevating graph. In this short section we provides some properties of the elevating graph and we find a (very large) lower bound for the number of transfer systems. With the exception of the lower bound, we only sketch the arguments of the other results. 

Let $B_n$ be the boolean lattice of the subsets of $[n]$. Then we have:
\begin{enumerate}
\item There are $3^n - 2^n$ join-irreducibles in $\Trs(B_n)$. Indeed, the numbers of relations $(A,B)$ with $A\subseteq B$ is $3^n$ and there are $2^n$ trivial relations. 
\item The number of edges in the elevating graph of $B_n$ is 
\[
\binom{3^n - 2^n}{2} - \big( 6^n - 5^n -3^n +2^n\big).
\]
We count the number of edges in the Galois graph, which is the complement of the elevating graph. There are $6^n$ commutative squares 
\[
\xymatrix{
A\ar[r]\ar[d] & C\ar[d]\\
B\ar[r] & D
}
\]
and $5^n$ commutative squares with $B\subseteq C$. So, there are $6^n - 5^n$ commutative squares where $(A,B)$ does not lift on the left $(C,D)$. Note that in such a square we have $A\neq B$ and $C\neq D$. It remains to remove the squares where $(A,B) = (C,D)$ and there are $3^n - 2^n$ such squares. 

This sequence starts with $0,0,4,99,1474,17715,190414,\ldots$
\end{enumerate}

Note that if there is an elevating set of size $k$, then any subset is also an elevating set. So if $k$ is the largest size of an elevating set, there are at least $2^k$ elavating sets. Moreover there are numerous upper bounds for the number cliques in a graph. Here we use the upper bound of \cite[Theorem 1]{wood2007maximum}. We recall that the largest size of an elevating set is given by $\maxcov(L)$, the maximal number of lower cover relations in the lattice of transfer systems. It follows that
\begin{proposition}
Let $(L,\leq)$ be a finite lattice. We denote by $k=\maxcov(L)$ and by $n$ the number of non-trivial relations in $L$. Then we have
\[
2^k \leq |\Trs(L)| \leq \sum_{j=0}^{k} \binom{k}{j} (\frac{n}{k})^j
\]
\end{proposition}
Let us focus on the boolean lattices $B_n$. For $0\leq k \leq 2n-1$, we consider the relation $\Rc_k$ on $B_n$ defined by
\[
X \leq_{\Rc_k} Y \hbox{ if and only if $X = Y$ or $X\subseteq Y$ and $|X|+|Y|\leq k$.}
\]
\begin{lemma}
The relation $\Rc_k$ is a transfer system for the boolean lattice $B_n$.
\end{lemma}
\begin{proof}
The relation is by definition reflexive, and it is clearly anti-symmetric. For transitivity let us assume that we have $X\leq_{\Rc_k}Y $ and $Y\leq_{\Rc_k}Z$. The first condition implies that $|X|\leq |Y|$, so $|X|+|Z|\leq |Y|+|Z|\leq k$ and $\Rc_k$ is a subposet of the inclusion relation. 

Let us consider a commutative diagram
\[
\xymatrix{
A \ar[r]\ar[d]& B\ar[d]\\
X\ar[r]& Y
}
\]
such that $X\leq_{\Rc_k}Y$, and $X\neq Y$. Then we have $|A|+|B|\leq |X|+|Y|\leq k$, so $A \leq_{\Rc_k} B$. This is in particular true, when the diagram is a pullback diagram. When $X = Y$, the only relations that we obtain by taking pullback along $X=X$ are the equalities relations. Hence $\Rc_k$ is a transfer system. 
\end{proof}
\begin{lemma}\label{lem:formaul_cov}
The number of lower cover relations of $\Rc_k$ is
\[
\sum_{j=0, j\neq \frac{k}{2}}^{n}\binom{n}{j}\binom{n-j}{k-2j}. 
\]
\end{lemma}
\begin{proof}
Let $X\lessdot Y$ be a cover relation. There are two possibilities:
\begin{itemize}
\item Cover relation of type 1: $|X|+|Y|=k$.
\item Cover relation of type 2: $|X|+|Y|<k$. 
\end{itemize}
First let us remark that if $X\subseteq Y$ and $|X|+|Y|=k$, then $X\leq_{\Rc_k} Y$ is a cover relation. Indeed, if we have $X\leq_{\Rc_k} X' \leq_{\Rc_k} Y$, then $k \geq |X'|+|Y|\geq |X|+|Y| = k$, so $X=X'$.

Let $X\lessdot Y$ be a cover relation of type $2$. First, we see that $Y$ has exactly one more element than $X$. Indeed, if $y \in Y\setminus X$, we have $|X|+|X\cup \{y \} |\leq |X|+|Y| <k$ and $|X\cup \{y \}|+|Y|=|X|+|Y|+1 \leq k-1+1 = k$, so we have $(X,X\cup \{y\})$ and $(X\cup \{y\},Y)$ are two relations of $\Rc_k$, so $Y = X\cup \{y\}$. We can find two subsets $X'$ and $Y'$ such that $X\subseteq X'$, $Y\subseteq Y'$, $X'\cap Y= X$ and $|X'|+|Y'|=k$. Indeed, we have $|Y|\leq \frac{k}{2} < n$ and if we add an element of $[n]\setminus Y$ to $Y$, we obtain a set $Y'$ such that $X\subseteq Y'$ and $|X|+|Y'|= |X|+|Y|+1\leq k$, hence a new relation $X\leq_{\Rc_k}Y'$. Repeating the process we can increasing $|X|+|Y'|$ up to $|X|+n$. We stop when $|X|+|Y'|=k$. If this is not enough, we can add the elements of $[n]\setminus Y$ to $X$. Since these elements are not in $Y$, the intersection of this new set and $Y$ is still equal to $X$. In the worst case, we end up with $X' = [n]\setminus\{y\}$ and $Y' = [n]$, so $|X'|+|Y'|=2n-1 \geq k$.

Then, there is a pullback diagram:
\[
\xymatrix{
X \pb \ar[r] \ar[d] & Y\ar[d]\\
X' \ar[r]_{\in \Rc_k^{c}} & Y' 
}
\]
so $X\lessdot Y$ is not a maximal cover relation. It remains to count the cover relations of type $1$. If $j$ is the size of $X$, then the size of $Y\setminus X$ is $k-2j$, so we have
\[
\sum_{j=0}^{n}\binom{n}{j}\binom{n-j}{k-2j}
\]
pairs of subsets $X\subseteq Y$ with $|X|+|Y|=k$. We have to remove the pairs where $X=Y$, they only occur when $k$ is even and they appear when $j=k/2$.  
\end{proof}
Removing $\Rc_{0}$, these numbers naturally form a triangle, see \cref{fig:triangle}. 
\begin{figure}[h!]
\begin{center}
\defn{1}\\
\defn{2} 1 \defn{2}\\
3 3 \defn{7} 3 3 \\ 
4 6 \defn{16} 13 \defn{16} 6 4\\
5 10 30 35 \defn{51} 35 30 10 5\\
6 15 50 75 \defn{126} 121 \defn{126} 75 50 15 6\\
7 21 77 140 266 322 \defn{393} 322 266 140 77 21 7
\end{center}
\caption{Triangle of the number of lower cover relations of the $\Rc_k$ for $k=1,\dots,2n-1$.}\label{fig:triangle}
\end{figure}
Note that the maximal numbers in each row (in red) appear in the expansion of $(1+x+x^2)^n$. When $n$ is odd, this is the coefficient of $x^n$ and when $n$ is even, it is the coefficient of $x^{n-1}$ or $x^{n+1}$. 

\begin{proposition}\label{pro:max_cov}
    Let $n\in \mathbb{N}$ and $L=B_n$ the boolean lattice of the subsets of $[n]$. Then
    \begin{enumerate}
        \item If $n$ is odd, $\maxcov(L)\geq \sum_{j=0}^{\frac{n-1}{2}} \binom{n}{j} \binom{n-j}{n-2j}$.
        \item If $n$ is even, $\maxcov(L)\geq \sum_{j=1}^{\frac{n}{2}} \binom{n}{j}\binom{n-j}{n+1-2j}.$
    \end{enumerate}
\end{proposition}
\begin{proof}
This is a special case of \cref{lem:formaul_cov}.
\end{proof}
\begin{question}\label{qu:bool}
    Let $L= B_n$ be the boolean lattice of the subsets of $[n]$. Is $\maxcov(\Trs(L))$ given by \cref{pro:max_cov} ?
\end{question} 

Using the computer, we were are able to check that \cref{qu:bool} has a positive answer for $n\leq 6$.

\begin{tabular}{|c|c|c|c|c|c|c|c|c|c|}\hline $n$ & $1$ & $2$ & $3$ & $4$ & $5$ &$6$ & $7$\\\hline$\maxcov $  & $1$ & $2$ & $7$ & $16$ & $51$ & $126$ & $\geq 393$\\\hline$2^{\maxcov} \approx$  & $2$ & $4$ & $128$ & $65536$ & $2.2\cdot 10^{15}$ & $8.5\cdot 10^{37}$& $\geq 2\cdot 10^{118}$ \\\hline $|\Trs(B_n)| $& $2$ & $10$ & $450$ & $5389480$ & ? & ? & ?\\ \hline $|spine|$ &2 & 8 & 288 & 1975552 & ? &? &? \\ \hline $\sum_{j=0}^{k} \binom{k}{j} (\frac{n}{k})^j \approx$ & 2 & 12 &9752 &$1.8\cdot 10^{11}$ &$1.8 \cdot 10^{36}$ &$3.3 \cdot 10^{100}$ & ? \\\hline\end{tabular}

\begin{remark}
The spine of the lattice of transfer systems on $B_n$ is isomorphic to the lattice of order ideals of the poset $(\Rel^*(B_n),\preceq)$. The number of order ideals is also the number of antichains. Moreover an antichain in the poset $(\Rel^*(B_n),\preceq)$ is also an antichain in the poset of \emp{all relations} $(\Rel(B_n),\preceq)$. It is not difficult to check that this largest poset is isomorphic to the product $[3]^{n}$ of total orders of length $3$. It is notoriously difficult to count antichains in such products (for product of total orders with $2$ elements, this is the so-called Dedekind's problem) however, there are some known upper bounds, see for example \cite{park2023note}. 
\end{remark}

\section{Dictionary: transfer systems and torsion classes}\label{sec:dico}
The lattice of the torsion pairs of the category of finite dimensional modules over a finite dimensional algebra shares most of the lattice properties of the lattice transfer systems on a finite lattice. Here we provide a dictionary and hope that it will help to navigate between the two subjects. 

{\small 
\begin{center}
\begin{tabular}{|l|l|l|ll}
\cline{1-3}
                   & Weak factorization systems            & Torsion Classes                              &  &  \\ \cline{1-3}
Finite             & Yes, by definition                    & Sometimes when $\tau$-tilting finite $(1)$         &  &  \\ \cline{1-3}
Regular (6)        & ?                                   & When finite                                  &  &  \\ \cline{1-3}
Semidistributive & Yes                                   & Yes  (2)                             &  &  \\ \cline{1-3}
Distributive    & Only when $n\leq 2$                   & Only when product of local algebras (7)                         &  &  \\ \cline{1-3}
Congruence uniform & Yes                                   & Yes (2)                             &  &  \\ \cline{1-3}
Trim               & Yes                                   & Sometimes $(3)$ &  &  \\ \cline{1-3}
Join irreducible   & Relations of the lattice              & Bricks $(4)$                                       &  &  \\ \cline{1-3}
Canonical joinands & Elevating sets                        & Semibricks $(5)$                                   &  &  \\ \cline{1-3}
Canonical complex  & Clique complex of the elevating graph & ?                                            &  &  \\ \cline{1-3}
\end{tabular}
\end{center}
}
\begin{enumerate}
    \item There are many equivalent characterizations of the algberas with finitely many torsion pairs. We refer for example to \cite[Theorem 10.2]{treffinger} for a large list of such characterizations. 
    \item The lattice of torsion pairs is completely semidistributive and completely congruence uniform \cite{DIRRT}.
    \item We don't think that there is a known classification of the finite dimensional algebra whose lattice of torsion pairs is trim, however this is the case for the path algebras of Dynkin quivers. More generally, the lattice of torsion pairs of the incidence algebras of finite representation type are trim by \cite{TW} and \cite[Theorem 1.4]{simson2010integral}. In that case, the spine consists of the so-called \emp{split torsion pairs}. Note that there are many examples of $\tau$-tilting finite algebras whose lattice of torsion pairs is not trim. 
    \item The completely join-irreducible elements of the lattice of torsion pairs are in bijection with the bricks in the module category by \cite[Theorem 1.4]{DIRRT}. Here by a brick we mean a module $X$ is whose endomorphism algebra is a division algebra. The join-labelling is in this case known under the name \emp{brick labelling}. 
    \item The set of bricks that are canonical join representations are called \emp{semibricks} and they were studied in \cite{asai} and also in \cite{BTZ}. Note that since the lattice of torsion pairs is most of the time infinite, the treatment is more subtle. For example, there are non-zero elements in the lattice without any lower cover relation.
    \item Here a poset is said to be \defn{regular} if its (unoriented) Hasse diagram is a regular graph. The lattice of weak factorization systems on a total order is isomorphic to the Tamari lattice, hence it is regular. We don't know any other example of regular lattice of weak factorization systems. On the other hand, when the lattice of torsion pairs is finite, there is a bijection between torsion pairs and (support) $\tau$-tilting modules. There is a \emp{mutation} theory for these objects which implies the regularity of the lattice.  \item The lattice of torsion pairs of a finite dimensional algebra is distributive if and only if it is boolean. In this case, this is the lattice of torsion pairs of a semisimple algebra. See \cite{luo2023boolean} for more details. 
\end{enumerate}
 \bibliographystyle{plain}
\bibliography{biblio.bib}
 \end{document}